\documentclass[a4paper,12pt,reqno]{amsart}
\usepackage[T1]{fontenc}
\usepackage[utf8]{inputenc}
\usepackage[libertinus,vvarbb]{newtx}
\usepackage[margin=1in]{geometry}
\usepackage{enumitem}
\usepackage{graphicx}
\usepackage[dvipsnames]{xcolor}
\usepackage{tikz}
\usepackage{xparse}
\usepackage{xurl}
\usepackage[hidelinks]{hyperref}
\usepackage{tikz}
\usepackage{xcolor}
\usepackage{cite}
\usepackage{changes} 

\definecolor{MyBlue}{RGB}{200, 225, 255}
\definecolor{MyGreen}{RGB}{200, 245, 200}
\definecolor{MyOrange}{RGB}{255, 225, 180}

\usetikzlibrary{decorations.markings,decorations.pathmorphing,nfold}
\makeatletter
\tikzset{
 side by side path/.style args={#1:#2}{
  postaction={path only,draw=#1,offset=+.5\pgflinewidth},
  postaction={path only,draw=#2,offset=+-.5\pgflinewidth}
 },
 side by side/.style={path only,side by side path={#1}},
 offset/.code=\tikz@addoption{%
  \pgfgetpath\tikz@temp%
  \pgfsetpath\pgfutil@empty%
  \pgfoffsetpath\tikz@temp{#1}%
 }
}
\makeatother

\numberwithin{equation}{section}

\newtheorem{theorem}{Theorem}[section]
\newtheorem{lemma}[theorem]{Lemma}
\newtheorem{corollary}[theorem]{Corollary}
\newtheorem{proposition}[theorem]{Proposition}

\theoremstyle{definition}

\DeclareMathOperator{\im}{Im}

\DeclareMathOperator{\arccot}{arccot}

\newcommand{\ind}{\mathbb{1}}

\newcommand{\C}{\mathbb{C}}
\newcommand{\R}{\mathbb{R}}
\newcommand{\Z}{\mathbb{Z}}

\newcommand{\disk}{\mathbb{D}}
\newcommand{\torus}{\mathbb{T}}

\newcommand{\ph}{\varphi}
\newcommand{\eps}{\varepsilon}
\newcommand{\thet}{\vartheta}

\mathchardef\mathhyphen="2D

\newcommand{\cm}{\mathscr{C}\mspace{-4mu}\mathscr{M}}
\newcommand{\am}{\mathscr{A}\mspace{-4mu}\mathscr{M}}
\newcommand{\amcm}{\am\mspace{-2mu}/\mspace{-1mu}\cm}

\renewcommand{\leq}{\leqslant}
\renewcommand{\geq}{\geqslant}

\newcommand{\lv}{\lvert}
\newcommand{\rv}{\rvert}
\newcommand{\lV}{\lVert}
\newcommand{\rV}{\rVert}

\newcommand{\floor}[1]{\left\lfloor #1 \right\rfloor}

\NewDocumentCommand{\formula}{ssom}{%
 \IfBooleanTF{#1}{%
  \IfBooleanTF{#2}{%
   \IfValueTF{#3}%
    {\begin{align}\label{#3}\begin{gathered}#4\end{gathered}\end{align}}%
    {\begin{gather}#4\end{gather}}%
  }{%
   \IfValueTF{#3}%
    {\begin{align}\label{#3}\begin{aligned}#4\end{aligned}\end{align}}%
    {\begin{gather*}#4\end{gather*}}%
  }%
 }{%
  \IfValueTF{#3}%
   {\begin{align}\label{#3}#4\end{align}}%
   {\begin{align*}#4\end{align*}}%
 }%
}
\begin{document}

\title[Eigenvalues of Hermitian Toeplitz matrices with Fisher--Hartwig symbols]{Eigenvalues of Hermitian Toeplitz matrices with Fisher--Hartwig symbols}
\author{Jacek Wszoła}
\thanks{Work supported by the Polish National Science Centre (NCN) grant no.\@ 2023/49/B/ST1/04303}
\address{\normalfont Jacek Wszoła \\ Department of Analysis and Stochastic Processes \\ Wrocław University of Science and Technology \\ Wybrzeże Wyspiańskiego 27 \\ 50-370 Wrocław, Poland}
\email{jacek.wszola@pwr.edu.pl}
\date{\today}

\begin{abstract}
We investigate the eigenvalues of Hermitian Toeplitz matrices generated by the symmetric Fisher--Hartwig symbol $\hat a(z) = (1-z)^{\alpha/2}(1-z^{-1})^{\alpha/2}$ for $\alpha \in (0,2)$. Using Dirichlet--Neumann bracketing of the discrete Laplacian, we establish explicit, non-asymptotic bounds for the individual eigenvalues. As a consequence, we prove that all eigenvalues are simple. We also obtain a two-term approximation for every
eigenvalue, with explicit bounds on the remainder, as the size of the matrix tends to infinity. While known results are limited to $\alpha > 1$, we bridge this gap by covering the full range of $\alpha \in (0,2)$. Our approach uses a construction of approximate eigenvectors that has not previously been applied in this setting.

\end{abstract}

\maketitle

\section{Introduction}

\subsection{Main results}

A Toeplitz matrix is a matrix in which each descending diagonal from left to right is constant. More precisely, given a sequence $(a_k: k \in \Z)$ in $\ell^2(\Z)$, we define the infinite Toeplitz matrix $T(a) = (a_{j-k}:j,k \geq 0)$. By $T_N(a) = (a_{j-k}: 0 \leq j,k \leq N-1)$ we denote the corresponding Toeplitz matrix of size $N$, which is a finite truncation of $T(a)$. The elements of the sequence $(a_k)$ are the Fourier coefficients of a certain function $\hat a \in L^2(\torus)$, where $\torus$ is the complex unit circle. The function $\hat a$ is often called the symbol of $T(a)$.

The spectral theory of Toeplitz matrices has been an area of intense interest for over a century, tracing its origins to the seminal works of Szeg\H{o} \cite{szego} in 1915. In recent decades, numerous papers have investigated the spectral properties (such as individual eigenvalues, singular values, eigenvectors, determinants or condition numbers) of truncated Toeplitz matrices $T_N(a)$ as their size $N$ tends to infinity. A comprehensive summary of results concerning the spectral theory of Toeplitz matrices and operators has been collected in the monographs by Böttcher and Silbermann \cite{bs, bs2}. Toeplitz matrices have also found applications in many areas such as stochastic processes, time series analysis, signal and image processing, quantum and statistical mechanics, or numerical solutions of differential and integral equations. For practical examples, see e.g. \cite{poland, dahlhaus, hno, dik, bottcher, serra}.

In recent years, symbols with power singularities have attracted particular interest. The most important examples in this class are symbols with a single Fisher--Hartwig singularity, given by
\formula[eq:fh:symbol]{
\hat a (z) = (1-z)^{\alpha/2} (1-z^{-1})^{\beta/2}
}
for $|z|=1$ and $\alpha, \beta \in (0,2)$. We will refer to \eqref{eq:fh:symbol} as the Fisher--Hartwig symbol. Symbols of this type were first introduced in \cite{fh}, in 1969, and they have been intensively studied over the years \cite{bw, dgk, kadanoff, bbg2023}. We refer the reader to \cite{bbg2023} for a comprehensive survey of recent results on the eigenvalues of matrices generated by Fisher--Hartwig symbols.

In this paper, we consider the class of Hermitian Toeplitz matrices generated by \eqref{eq:fh:symbol} with $\alpha = \beta$ so that the corresponding symmetric Fisher--Hartwig symbol reads
\formula[eq:G]{
g(\thet) = \hat a (e^{i\thet}) = \biggl( 2 \sin \frac{\thet}{2} \biggr)^{\alpha}
}
for $0 \leq \thet \leq 2\pi$.

In \cite{bg2022}, the authors investigated the asymptotic behaviour of eigenvalues for a broader class of symbols, known as \textit{simple loop} symbols. In their framework, the symbol is represented by \eqref{eq:G} with an additional analytic component $g_0(z)$. Our case corresponds to the situation where $g_0(z)$ is a constant function. To study the eigenvalues, the authors use the so-called \textit{simple loop method} which was introduced in \cite{bgm2010} and later developed in \cite{bbgm2015, bbgm2016, bbgm2017, bgmmrs2019, bg2022}.  The results in~\cite{bg2022}, however, cover only the parameter range $\alpha \in (1,2)$, and hold only for the eigenvalues with indices $n \geq \eps N$ for some $\eps>0$. In the present work, we extend these results to the entire interval $\alpha \in (0,2)$. Moreover, we eliminate the index threshold restrictions.

Our first result is the following theorem on separation of the eigenvalues, found using Dirichlet--Neumann bracketing.

\begin{theorem}
\label{thm:main:2}
Let $T_N(a)$ be the Toeplitz matrix corresponding to the symmetric Fisher--Hartwig symbol with $\alpha \in (0,2)$. Denote by $\lambda_n$ the eigenvalues of $T_N(a)$ arranged in a non-decreasing order. Then
\formula{
\biggl( 2 \sin \frac{(n-1)\pi}{2N} \biggr)^{\alpha} \leq \lambda_n \leq \biggl( 2\sin \frac{n\pi}{2(N+1)}\biggr)^{\alpha}
}
for all $1 \leq n \leq N$. In particular, all $\lambda_n$ are simple.
\end{theorem}

In a recent paper by Kwaśnicki and the author \cite{kw2026}, a generalized eigenvector expansion was established for a class of (non-Hermitian) infinite Toeplitz matrices $T(a)$, where sequences $(a_k: k\geq 1)$ and $(a_{-k}: k\geq1)$ are completely monotone. In our sign convention, these results are applied to the sequence $(-a_k)$ and the corresponding Toeplitz matrix $-T(a)$; see Section \ref{sec:expansion}. This expansion inspired our method for constructing approximate eigenvectors. Specifically, our approach relies on gluing the eigenvectors in a smooth way. This technique is a direct adaptation of the method already known in the literature from e.g. \cite{kkms2010, kwasnicki2012, kkm2013, kkm2016}, where it was successfully used to derive the eigenvalue asymptotics for the fractional Laplacian and related operators on an interval. However, to the best of the author's
knowledge, these methods have not been used to study the eigenvalues of Toeplitz matrices.

The next result provides an asymptotic formula for the eigenvalues in terms of the solutions to an implicit equation, where the error depends on $n$ and $N$.

\begin{theorem}
\label{thm:main}
Consider the following phase equation:
\formula[eq:arg:main]{
(N+1) \thet + \eta(\thet) = n\pi,
}
where
\formula{
\eta (\thet) & = \frac{2}{\pi} \int_0^\infty \frac{t}{1+t^2} \, \frac{\cos (\thet/2)}{\sqrt{t^2+\sin^2 (\thet/2)}} \, \arccot \frac{t^\alpha - \cos(\alpha\pi/2)}{\sin(\alpha\pi/2)} \, dt.
}
Under the assumptions of Theorem \ref{thm:main:2}, for $N \geq 2$ and $1 \leq n \leq N$,
\begin{enumerate}
    \item\label{it:1} the phase equation \eqref{eq:arg:main} has exactly one solution $\thet_n \in (0,\pi)$; 
    \item\label{it:2} we have the following asymptotics:
    \formula[eq:main]{
    \lambda_n = g(\thet_n) + E_1(n,N),
    }
    where
    \formula[eq:g]{
    E_1(n,N) = \mathscr{O} \biggl[ \frac{1}{(2-\alpha)\alpha^{3/2}}\, \frac{1}{nN^{\alpha}}\biggl(1-\frac{n}{N+1}\biggr) \biggr].
    }
\end{enumerate}
\end{theorem}

Phase equations with a remainder term also appear in the literature on eigenvalue asymptotics for Toeplitz matrices. In particular, Theorem~1(ii) of \cite{bg2022} with function $\eta_{\mathrm{BG}}$ defined in Eq. (5) therein, specialised to the symmetric Fisher--Hartwig symbol, gives an equation of this type. The following consequence of Theorem~\ref{thm:main} extends that result to all $\alpha\in(0,2)$ and all $1\leq n\leq N$, with an explicit bound on the remainder.

\begin{corollary}
\label{cor:equation}
    Let $s_n = g^{-1}(\lambda_n)$, where $g^{-1}$ denotes the inverse of the restriction of $g$ to $[0,\pi]$. Under the assumptions of Theorem~\ref{thm:main}, for $1 \leq n \leq N$,
    \formula{
    (N+1) s_n + \eta(s_n) = n \pi + E_2(n,N),
    }
    where
    \formula{
    E_2(n,N) = \mathscr{O} \biggl( \frac{1}{(2-\alpha) \alpha^{5/2}} \, \frac{1}{n^\alpha} \biggr).
    }
    In particular, if $\alpha \in (1,2)$ and $n \geq \eps N$ for some $\eps \in (0,1)$, then the remainder becomes $\mathscr{O}(N^{-\alpha})$. Moreover, functions $\eta$ and $\eta_{\mathrm{BG}}$ coincide on $[0,\pi]$. 
\end{corollary}

Our main results also yield a two-term approximation for the eigenvalues $\lambda_n$, valid for all $\alpha\in(0,2)$ and all $1\leq n\leq N$. In the symmetric Fisher--Hartwig case, this extends the parameter and index ranges covered by Theorem~3 of \cite{bg2022}.

\begin{theorem}
\label{thm:main:4}
    Let $U_{n,N} = n\pi/(N+1)$. Under the assumptions of Theorem \ref{thm:main}, for $1 \leq n\leq N$,
    \formula{
    \lambda_n = g(U_{n,N}) - \frac{g'(U_{n,N}) \eta(U_{n,N})}{N+1} + E_1(n,N) + E_3(n,N),
    }
    where $E_1(n,N)$ is defined in \eqref{eq:g} and
    \formula{
    E_3(n,N) = \mathscr{O}\biggl( \frac{\alpha n^{\alpha-2}}{N^\alpha} \biggr).
    }
    In particular, if $\alpha \in (1,2)$ and $n \geq \eps N$ for some $\eps \in (0,1)$, then the remainder becomes $\mathscr{O}(N^{-2})$.
\end{theorem}

Lastly, using the asymptotic expansion of the function $\eta$ that appears in Theorem \ref{thm:main}, we establish the approximation formulas for extreme eigenvalues at the edges of the spectrum. Below $C$ denotes a positive constant, independent of $\alpha, n$ and $N$.

\begin{theorem}
\label{thm:main:3}
Under the assumptions of Theorem \ref{thm:main},
\begin{enumerate}
    \item for $n \leq N/2$,
    \formula{
    \lambda_n =  g\biggl(\frac{(4n -  2+\alpha)\pi}{4N+2\alpha}\biggr) + E_1(n,N) - E_4(n,N),
    }
    where $E_1(n,N)$ is defined in \eqref{eq:g} and
    \formula{
    0 < E_4(n,N) \leq C\alpha \begin{cases}
        \dfrac{1}{1-\alpha} \, \dfrac{n^{2\alpha}}{N^{2\alpha+1}}\biggl(1-\dfrac{n}{N+1}\biggr) &\text{for } \alpha \in (0, 1), \\
        \dfrac{ n^2}{N^3} \log \dfrac{N}{n} \biggl( 1- \dfrac{n}{N+1}\biggr)  &\text{for } \alpha = 1, \\
        \dfrac{1}{\alpha-1} \, \dfrac{n^{\alpha+1}}{N^{\alpha+2}} \biggl( 1- \dfrac{n}{N+1}\biggr) &\text{for } \alpha \in (1,2); \\
    \end{cases}
    }
    \item for $n > N/2$,
    \formula{
    \lambda_n = g\biggl( \frac{(n-C_\alpha)\pi}{N+1-C_\alpha}\biggr) + E_1(n,N) - E_5(n,N),
    }
    where
    \formula{
     C_\alpha &= \frac{1}{\pi} \int_0^\infty \frac{t}{(1+t^2)^{3/2}} \arccot \frac{t^\alpha - \cos (\alpha\pi/2)}{\sin(\alpha\pi/2)}\,dt,
    }
    $E_1(n,N)$ is defined in \eqref{eq:g}, and 
    \formula{
    0 < E_5(n,N) &\leq C\alpha \frac{n^{\alpha-1}}{N^\alpha} \biggl( 1-\frac{n}{N+1}\biggr)^4.
    }
\end{enumerate}
\end{theorem}

\subsection{Organisation}
The paper is organised as follows. In Section \ref{sec:preliminaries}, we discuss properties of symmetric Fisher-Hartwig symbols and recall some properties of self-adjoint operators on $\ell^2(\Z)$. In Section \ref{sec:laplacians} we define the discrete Laplacian and its Dirichlet and Neumann versions. We then use these operators to prove Theorem~\ref{thm:main:2}. Section \ref{sec:eigenvectors} is devoted to discussing eigenvector expansion of Toeplitz matrices corresponding to symmetric Fisher--Hartwig symbols. We also prove some auxiliary results. Construction of the approximate eigenvectors is discussed in Section~\ref{sec:approximate:eigenvectors}. There we also introduce and study the phase equation \eqref{eq:arg:main}. Section~\ref{sec:proofs} contains the proofs of Theorems \ref{thm:main}, \ref{thm:main:4} and \ref{thm:main:3} and Corollary~\ref{cor:equation}. Finally, in Section \ref{sec:simulations}, we present the results of numerical analysis, including comparison of the estimates for the eigenvalues.

\section{Preliminaries}
\label{sec:preliminaries}

\subsection{Definitions and notations}

In this paper, $\Z = \{\ldots, -2,-1,0,1,2,\ldots\}$ denotes the set of integers. The letters $j,k,l,n,m$ and $N$ are reserved to represent integers. By a sequence we mean either a one-sided sequence $(a_k: k \geq 1)$ or a two-sided sequence $(a_k: k \in \Z)$. If the meaning is clear from the context, we write simply $(a_k)$. For a sequence $(a_k)$ we define its difference sequence $\Delta a_k = a_{k+1}-a_k$, where $\Delta$ is a forward difference operator. The $m$th iterated difference of $(a_k)$ is defined inductively by $\Delta^m a_k = \Delta\Delta^{m-1} a_k$, with $\Delta^0 a_k=a_k$, or directly $\Delta^m a_k = \sum_{j=0}^m \binom{m}{j} (-1)^{m-j} a_{k+j}$.

We say that $(a_k:k \geq 1)$ is completely monotone if $(-1)^m \Delta^m a_k\geq 0$ for every $m \geq 0$ and $k \geq 1$. Furthermore, if $(a_{-k}: k \geq 1)$ is completely monotone, then $(a_k: k \leq -1)$ is said to be absolutely monotone. For a sequence $(a_k: k \in \Z)$, if both $(a_k: k \geq 1)$ and $(a_{-k}: k \geq 1)$ are completely monotone, then we say that $(a_k)$ is an \textit{$\amcm$ (absolutely monotone then completely monotone) sequence}. If $a_k = a_{-k}$ for every $k \geq 1$ we say that $(a_k)$ is symmetric. The notion of $\amcm$ sequences originated in \cite{kw2025} and was later used in \cite{kw2026} to generate a class of non-Hermitian Toeplitz matrices. For further discussion on the $\amcm$ class, see Section 3 in \cite{kw2026}.

Following the standard notation, by $\ell^2(\Z)$ we denote the Hilbert space of two-sided real-valued square-summable sequences, that is $\smash{\sum_{k \in \Z} a_k^2 < \infty}$. For $(a_k), (b_k) \in \ell^2(\Z)$ we define the norm of $(a_k)$ by $\smash{\lV a \rV = (\sum_{k \in \Z} a_k^2)^{1/2}}$ and the corresponding inner product for real-valued sequences is given by $\smash{\langle a, b \rangle = \sum_{k \in \Z} a_k b_k}$.

Unless explicitly defined otherwise, $C$ denotes a positive constant. The value of $C$ may change from line to line. We use standard big O notation: we write $f(n,N) = \mathscr{O}(g(n,N))$ if there is a constant $C>0$ such that $|f(n,N)| \leq C g(n,N)$ for all admissible $n,N$. Moreover, we write $f(n,N) \approx g(n,N)$ if there are constants $C_1, C_2>0$ such that $C_1 g(n,N) \leq f(n,N) \leq C_2g(n,N)$.

\subsection{Fisher--Hartwig symbol}

For a two-sided sequence $(a_k: k \in \Z)$, we let $T(a) = (a_{k-l}: k,l \geq 0)$ be a Toeplitz operator (an infinite Toeplitz matrix) associated with $(a_k)$. By $T_N(a)$ we denote the truncation of $T(a)$ to the $N\times N$ matrix with rows and columns indexed by $\{0,1,2,\ldots,N-1\}$. In this paper, we are interested in Hermitian Toeplitz matrices with Fisher--Hartwig symbol $\hat a$ defined in \eqref{eq:fh:symbol}, where
\formula{
a_k = \frac{1}{2\pi} \int_0^{2\pi} \hat a(e^{i\thet}) e^{ik\thet} \, d\thet
}
for $k \in \Z$. From now on, by $(a_k)$ we denote the sequence with the generating function $\hat a$. If it is clear from the context, we will write $T$ and $T_N$ instead of $T(a)$ and $T_N(a)$, respectively.

The sequence $(a_k)$ is symmetric. Moreover, for $k \geq 0$ we have
\formula[eq:a]{
a_k = \frac{(-1)^{k} \Gamma(1+\alpha)}{\Gamma(1+k+\alpha/2)\Gamma(1-k+\alpha/2)},
}
where $\Gamma(x)$ is the gamma function (see Lemma 3.2 in \cite{bg2022}). It is easy to verify that $(-a_k: k \geq 1)$ and $(-a_{-k}: k \geq 1)$ are completely monotone sequences, and hence $(-a_k)$ is an $\amcm$ sequence. Since $\hat a(1)=0$, we also have $a_0 =-\sum_{k \neq 0}a_k$.

\begin{lemma}
\label{lem:ak:bounds}
    For $(a_k)$ defined in \eqref{eq:a} and $k \geq 1$, we have
    \formula{
    &-\!a_k \leq \frac{C\alpha}{k^{\alpha+1}},&& \Delta a_k \leq \frac{C\alpha}{k^{\alpha+2}}, & -\Delta^2 a_k \leq \frac{C\alpha}{k^{\alpha+3}}. &
    }
\end{lemma}

\begin{proof}
Observe that
\formula{
\frac{1}{\Gamma(1-k+\alpha/2)} = \frac{ \Gamma(k-\alpha/2) \sin((k-\alpha/2)\pi)}{\pi} = \frac{(-1)^{k+1}\Gamma(k-\alpha/2) \sin (\alpha \pi/2)}{\pi},
}
Hence, \eqref{eq:a} can be simplified as follows:
\formula{
-a_k = \frac{\Gamma(1+\alpha) \sin (\alpha\pi/2)}{\pi} \, \frac{\Gamma(k-\alpha/2)}{\Gamma(1+k+\alpha/2)}. 
}
Recall the asymptotics of the ratios of the gamma function, that is $\Gamma(z)/\Gamma(z+\gamma) \leq Cz^{-\gamma}$ for $z\geq z_0>0$ and $\gamma>0$ (see 6.1.47 in \cite{abramowitz}). For $k\geq2$, we have $k-\alpha/2\geq k/2$, and the gamma-ratio estimate gives
\formula{
-a_k \leq \frac{C\alpha}{(k-\tfrac{\alpha}{2})^{1+\alpha}} \leq \frac{C\alpha}{k^{1+\alpha}}.
}For $k=1$, we use \eqref{eq:a}:
\formula{
-a_1=\frac{\Gamma(1+\alpha)}{\Gamma(2+\alpha/2)\Gamma(\alpha/2)} = \frac{\alpha}{2}\frac{\Gamma(1+\alpha)}{(1+\alpha/2)\Gamma(1+\alpha/2)^2}\leq C\alpha.
}
Finally, the recurrence relation for the gamma function yields
\formula{
\frac{a_{k+1}}{a_k} = \frac{k-\alpha/2}{k+1+\alpha/2}.
}
Hence, for $k \geq 1$,
\formula{
\Delta a_k= -\biggl(1 -\frac{a_{k+1}}{a_k} \biggr)a_k = -\frac{1+\alpha}{k+1+\alpha/2}\,a_k\leq -\frac{C}{k}a_k\leq\frac{C\alpha}{k^{\alpha+2}},
}
and
\formula{
-\Delta^2a_k = -\biggl(\frac{a_{k+2}}{a_k}-2\frac{a_{k+1}}{a_k}+1\biggr)a_k = -\frac{(\alpha+1)(\alpha+2)}{(k+1+\alpha/2)(k+2+\alpha/2)}a_k \leq \frac{C\alpha}{k^{\alpha+3}},
}
as desired.
\end{proof}

\subsection{Wiener--Hopf factorisation and complete Bernstein functions}

A \textit{Wiener--Hopf factorisation} of a function $f$ on $\torus = \{z \in \C: |z|=1\}$ is a representation
\formula{
f(z)=f^+(z)f^-(z^{-1}),
}
where $f^+$ and $f^-$ are holomorphic and non-vanishing in the unit disk $\disk = \{z \in \C: |z|<1\}$, with boundary values for which the identity holds. The exterior factor is $z\mapsto f^-(z^{-1})$. If both factors extend continuously and without zeros to $\smash{\overline{\disk}}$, they are unique up to the replacement $(f^+,f^-)\mapsto(cf^+,c^{-1}f^-)$ for $c \neq 0$. For the difference quotients considered in Section \ref{sec:wk}, which vanish at $z=1$, we use the specific construction in \cite{kw2026}.

Wiener--Hopf factorisation is a standard tool in the study of Toeplitz matrices. For details, see e.g. Chapter 1.5 in \cite{bs}.

A function $f$ is called a \textit{complete Bernstein function} if it is holomorphic on $\C \setminus (-\infty,0]$, non-negative on $(0,\infty)$, and maps the upper complex half-plane into its closure. Every complete Bernstein function has an integral representation
\formula[eq:cbf]{
f(z) = c_1z + c_0 + \frac{1}{\pi} \int_{(0,\infty)} \frac{z}{z+s} \, \nu(ds),
}
where $c_0,c_1\geq 0$ and $\nu$ is a non-negative measure on $(0,\infty)$ such that $\min\{1,s^{-1}\}$ is integrable with respect to $\nu$. By \eqref{eq:cbf}, every complete Bernstein function is increasing on $(0,\infty)$. For further discussion, we refer to Chapters 6--7 in \cite{ssv}.

\subsection{Properties of Hermitian operators} Below we collect some classical results concerning Hermitian operators. Even though these results hold for a general Hilbert space setup, we will formulate the statements for $\ell^2(\Z)$ considered in this paper. Throughout this subsection, operators on $\ell^2(\Z)$ are assumed to be bounded and self-adjoint.

Let $A,B$ be Hermitian operators on $\ell^2(\Z)$. We write $A \leq B$ if for any $v \in \ell^2(\Z)$ we have $\langle Av, v \rangle \leq \langle Bv, v \rangle$. A continuous function $f:[0,\infty)\to\R$ is called \textit{operator monotone} if $f(A)\leq f(B)$ whenever $A$ and $B$ are bounded self-adjoint operators satisfying $0\leq A\leq B$. Here $f(A)$ and $f(B)$ are defined by means of the spectral theorem. Similarly, we say that a continuous function $f:[0,\infty) \to \R$ is \textit{operator concave} if for all $\lambda \in [0,1]$ and non-negative Hermitian operators $A,B$ on $\ell^2(\Z)$ we have $(1-\lambda)f(A) + \lambda f(B) \leq f((1-\lambda)A+\lambda B)$.

\begin{lemma}
\label{lem:operator:concave}
    If $\alpha \in (0,2)$, then $f(x) = x^{\alpha/2}$ is both operator monotone and operator concave. 
\end{lemma}
\begin{proof}
    By Theorem V.1.9 in \cite{bhatia}, $f$ is operator monotone. Since $f$ is continuous and maps $[0,\infty)$ into itself, by Theorem V.2.5 in \cite{bhatia}, $f$ is also operator concave.
\end{proof}

For $E \subset \Z$, we define $\ind_E$ to be the projection operator onto the subspace $\{v \in \ell^2(\Z): v_k=0 \text{ for } k \in \Z \setminus E\}$. More precisely, for $v \in \ell^2(\Z)$, we set
\formula{
(\ind_E v)_k = \begin{cases}
    v_k &\text{for } k \in E, \\
    0 &\text{for } k \not\in E.
\end{cases}
}
In the literature, one can also find the operator counterpart of Jensen's inequality. We state a simplified version
below. For the extended version see Theorem V.2.3 in \cite{bhatia}.

\begin{theorem}[Jensen's operator inequality]
\label{thm:jensen}
Let $f:[0,\infty)\to\R$ be a continuous operator concave function satisfying $f(0)\geq0$, and let $E\subset\Z$. If $A$ is a bounded non-negative self-adjoint operator on $\ell^2(\Z)$, then
\formula{
\ind_E f(A)\ind_E\leq f(\ind_E A\ind_E).
}
\end{theorem}

For a self-adjoint operator $A$ on $H_N=\{v \in \ell^2(\Z): v_k=0 \text{ for } k \not\in \{0,1,\ldots,N-1\}\}$, let $\lambda_1(A)\leq\cdots\leq\lambda_N(A)$ denote its eigenvalues, counted with multiplicity. The following finite-dimensional minimax principle follows from Corollary~III.1.2 in \cite{bhatia}, after identifying $H_N$ with $\mathbb R^N$ through an orthonormal basis.

\begin{theorem}[The minimax principle]
    Let $A:H_N\to H_N$ be self-adjoint. Then, for $1\leq n\leq N$, 
    \formula{
    \lambda_n(A)=\min\bigl\{\max\{\langle Av,v\rangle:v\in V,\ \|v\|=1\}:V\text{ is an }n\text{-dimensional subspace of }H_N\bigr\}.
    }
\end{theorem}

Consequently, if $A,B:H_N\to H_N$ are self-adjoint and $A\leq B$, then for $1 \leq n \leq N,$
\formula[eq:lambdas]{
\lambda_n(A) \leq \lambda_n(B).
}

\section{Spectral bounds}

\subsection{Discrete Laplacians}
\label{sec:laplacians}

Let $L$ be the discrete Laplacian on $\ell^2(\Z)$, defined by
\formula{
(Lu)_k = u_{k-1}-2u_k+u_{k+1}
}
for $k \in \Z$. Then $-L$ is a bounded non-negative self-adjoint operator with spectrum $[0,4]$. For $N\geq1$, we set $E_N=\{0,1,\ldots,N-1\}$ and let $H_N$ be the subspace of $\ell^2(\Z)$ consisting of sequences supported on $E_N$. We identify $H_N$ with $\R^N$ and denote the orthogonal projection onto $H_N$ by $\ind_N$.

For $k\in E_N$, let
\formula{
n_N(k)=\#\{j\in E_N:|j-k|=1\}.
}
We use the same notation for the multiplication operator $(n_Nu)_k=n_N(k)u_k$ on $H_N$. We also define the matrix
\formula{
A_N(j,k)=\begin{cases}-1 &\text{if } |j-k|=1,\\ 0 &\text{otherwise},\end{cases}
}
for $j,k \in E_N$. Moreover, denote by $I_N = (\ind_{j=k}: j,k \in E_N)$ the identity matrix in $H_N$. The Dirichlet and Neumann Laplacians on $E_N$ are defined by
\formula{
& -\!L_N^{\mathrm D}=2I_N+A_N, & -L_N^{\mathrm N}=n_N+A_N. &
}
In particular,
\formula{
-L_N^{\mathrm D}=\ind_N(-L)\ind_N \bigg\vert_{H_N}.
}
Our Neumann Laplacian agrees with the convention in Section~5.2 of \cite{kirsch}, whereas our Dirichlet Laplacian corresponds to the restriction $(H_0)_{E_N}$ used there.

The eigenvalues of these operators, arranged in increasing order, are
\formula[eq:laplacian:eigenvalues]{
& \lambda_n(-L_N^{\mathrm N})=4\sin^2\frac{(n-1)\pi}{2N}, & \lambda_n(-L_N^{\mathrm D})=4\sin^2\frac{n\pi}{2(N+1)}, &
}
where $1\leq n\leq N$. For the Neumann eigenvalues, see the matrix $A_2$ and the corresponding eigenvalue calculation in \cite{strang1999}; the Dirichlet formula follows from Proposition 2.1 of \cite{kst}.

\subsection{Dirichlet--Neumann bracketing}
\label{sec:proof:thm:2}

\begin{proof}[Proof of Theorem~\ref{thm:main:2}]
We first establish the following operator inequalities on $H_N$:
\formula[eq:bracketing]{
(-L_N^{\mathrm N})^{\alpha/2}\leq \ind_N(-L)^{\alpha/2}\ind_N\bigg\vert_{H_N}\leq(-L_N^{\mathrm D})^{\alpha/2}.
}
For the lower bound, extend $\smash{-L_N^{\mathrm N}}$ to an operator $B_N$ on $\ell^2(\Z)$ by setting it equal to zero on $H_N^\perp$. For every $u\in\ell^2(\Z)$, we have
\formula{
0 \leq \langle B_Nu,u\rangle = \sum_{k=0}^{N-2}(u_{k+1}-u_k)^2 \leq \sum_{k\in\Z}(u_{k+1}-u_k)^2 = \langle(-L)u,u\rangle.
}
Thus, $0\leq B_N\leq-L$. By the operator monotonicity of $x\mapsto x^{\alpha/2}$ established in Lemma~\ref{lem:operator:concave}, we obtain $\smash{B_N^{\alpha/2}\leq(-L)^{\alpha/2}}$. Compressing this inequality to $H_N$ gives the lower bound in \eqref{eq:bracketing}.

For the upper bound, the operator concavity of $x\mapsto x^{\alpha/2}$ and Jensen's operator inequality from Theorem~\ref{thm:jensen} yield
\formula{
\ind_N(-L)^{\alpha/2}\ind_N\leq\bigl(\ind_N(-L)\ind_N\bigr)^{\alpha/2}.
}
Restricting both sides to $H_N$ gives the upper bound in \eqref{eq:bracketing}.

Under the Fourier transform, $-L$ is multiplication by $2-2\cos\thet$. Consequently, $(-L)^{\alpha/2}$ has symbol $(2-2\cos\thet)^{\alpha/2}=g(\thet)$, and hence
\formula{
\left.\ind_N(-L)^{\alpha/2}\ind_N\right|_{H_N}=T_N.
}
By \eqref{eq:lambdas}, \eqref{eq:bracketing} and the spectral theorem, we obtain
\formula{
\bigl(\lambda_n(-L_N^{\mathrm N})\bigr)^{\alpha/2}\leq\lambda_n(T_N)\leq\bigl(\lambda_n(-L_N^{\mathrm D})\bigr)^{\alpha/2}.
}
Using \eqref{eq:laplacian:eigenvalues}, we conclude that
\formula{
\biggl(2\sin\frac{(n-1)\pi}{2N}\biggr)^\alpha\leq\lambda_n(T_N)\leq\biggl(2\sin\frac{n\pi}{2(N+1)}\biggr)^\alpha.
}
Finally, for $1\leq n<N$, these bounds imply
\formula{
\lambda_n(T_N)\leq\biggl(2\sin\frac{n\pi}{2(N+1)}\biggr)^\alpha<\biggl(2\sin\frac{n\pi}{2N}\biggr)^\alpha\leq\lambda_{n+1}(T_N).
}
Thus, all eigenvalues of $T_N$ are simple.
\end{proof}

\section{Eigenvectors and the phase equation}
\label{sec:eigenvectors}

\subsection{Eigenvector expansion}
\label{sec:expansion}
Generalised eigenvector expansion for Toeplitz matrices with $\amcm$ entries was the main object of study in \cite{kw2026}. This work provides a detailed description of generalised eigenvectors of this class of Toeplitz matrices. Even though the paper covers many non-Hermitian cases, we summarise these results in the special case of the symmetric Fisher--Hartwig symbol (see Section~7.6 in \cite{kw2026}).

Theorem 1.1 in \cite{kw2026} applied to the Fisher--Hartwig symbol states that for every polynomial $P$ the entries of $P(T)$ are given by
\formula{
(P(T))_{k,l} = \frac{2}{\pi} \int_0^{\pi} P(g(\thet)) v_{\thet,k} v_{\thet,l} \, d\thet,
}
where
\formula[eq:v]{
v_{\thet,k} = \sin\bigl((k+1)\thet + \tfrac{1}{2}\eta(\thet)\bigr)-w_{\thet,k}
}
for $k \geq 0$, where $\eta(\thet)$ is defined in Theorem~\ref{thm:main} and $(w_{\thet,k}: k \geq 0)$ is a completely monotone sequence. Definitions and properties of $\eta(\thet)$ and $(w_{\thet,k})$ are discussed in detail in Sections \ref{sec:eta} and \ref{sec:wk}, respectively. Theorem 1.2 in \cite{kw2026} implies that $v_{\thet}$ are true eigenvectors of the infinite Toeplitz matrix $T$, that is
\formula{
Tv_\thet = g(\thet) v_\thet.
}

In order to align with the literature on Toeplitz matrices, we have modified the notation used in \cite{kw2026}. In particular, we use $\thet = 2\arcsin ( \lambda^{1/\alpha}/2 )$ as a parameter instead of $\lambda$, and our $\eta(\thet)$ corresponds to $2 \thet_\lambda^+$. Note that $\lambda = g(\thet)$. Since our matrix is symmetric, the vectors $\smash{v^\pm_{\lambda,k}}$ defined in \cite{kw2026} coincide, allowing us to omit the superscripts to simply write $\smash{v_{\thet,k}}$ (and similarly, we write $\smash{w_{\thet,k}}$ for $\smash{w_{\thet,k}^\pm}$). The results of \cite{kw2026} are applied to the sequence $(-a_k)$, whose Toeplitz matrix is $-T(a)$. The function denoted by $F$ in those formulas corresponds to $-\hat a$ in our notation.

\subsection{Definition and properties of \texorpdfstring{$\eta$}{eta(theta)}}
\label{sec:eta}

The function $\eta(\thet)$ that appears in \eqref{eq:v} is given by
\formula{
\eta(\thet) &= \frac{2}{\pi} \int_0^1 \frac{\sin \thet}{(1-s)^2 + s(2\sin (\thet/2))^2}  \, \arccot \frac{g(\thet) s^{\alpha/2}(1-s)^{-\alpha} - \cos(\alpha\pi/2)}{\sin(\alpha\pi/2)} \, ds;
}
see Eq. (7.4) in \cite{kw2026}. It is convenient to substitute $t = 2\sin(\thet/2) s^{1/2} (1-s)^{-1}$ in the above integral. After simplification we get
\formula[eq:eta]{
\eta (\thet) & = \frac{2}{\pi} \int_0^\infty \frac{t}{1+t^2} \, \frac{\cos(\thet/2)}{\sqrt{t^2+\sin^2 (\thet/2)}} \, p_\alpha(t) \, dt,
}
where we denoted
\formula{
p_\alpha(t) = \arccot \frac{t^\alpha - \cos(\alpha\pi/2)}{\sin(\alpha\pi/2)}.
}
Furthermore, by Eq. (1.14) in \cite{kw2026}, if $\thet \in (0, \pi)$, then
\formula[eq:thetas:inequality]{
\thet \leq \thet + \eta(\thet) \leq \pi.
}
We now investigate the properties of $\eta(\thet)$, starting from the following auxiliary lemma.

\begin{lemma}
\label{lem:integral}
    We have
    \formula{
    \frac{1}{\pi} \int_0^\infty \frac{p_\alpha(s)}{1+s^2} \, ds = \frac{(2-\alpha)\pi}{8}.
    }
\end{lemma}

\begin{proof}
    Denote the left-hand side by $I$. Substituting $t = 1/s$ yields
    \formula{
    I = \frac{1}{\pi} \int_0^\infty \frac{p_\alpha(t^{-1})}{1+t^2} \, dt.
    }
    It follows that
    \formula{
    2I = \frac{1}{\pi} \int_0^\infty \frac{p_\alpha(s) + p_\alpha(s^{-1})}{1+s^2} \,ds.
    }
    Using the formula $\arccot x + \arccot y = \arccot((xy-1)/(x+y))$ for $x+y> 0$, we get, after simplification,
    \formula{
    2I = \frac{1}{\pi} \int_0^\infty \frac{1}{1+s^2} \arccot\biggl(-\cot \frac{\alpha \pi}{2}\biggr) \, ds = \frac{2-\alpha}{2} \int_0^\infty \frac{1}{1+s^2} \, ds = \frac{(2-\alpha)\pi}{4}.
    }
    Thus, $I = (2-\alpha)\pi/8$. 
\end{proof}

Our next result provides upper and lower bounds for $\eta(\thet)$ and the derivative $\eta'(\thet)$.

\begin{proposition}
\label{lem:theta:plus}
The function $\eta(\thet)$ is decreasing for $\thet \in (0,\pi).$ Moreover,
\formula[eq:theta:plus:bounds]{
    & 0 \leq \eta(\thet) \leq \frac{(2-\alpha)\pi}{4}, & -\frac{2-\alpha}{2} \leq \eta'(\thet) \leq 0.&
}
\end{proposition}

\begin{proof}
    The monotonicity of $\eta$ follows immediately from \eqref{eq:eta}. To prove the bounds of $\eta$, we need to find its limits at $0^+$ and $\pi^-$. Clearly, $\eta(\thet) \geq 0$ for $\thet \in (0,\pi)$. Hence, the limit at $\pi^-$ is 0, which follows from \eqref{eq:thetas:inequality}. More precisely,
    \formula{
    0 \leq \eta(\thet) \leq \pi-\thet.
    }
    For the limit at $0^+$, we use the monotone convergence theorem combined with Lemma \ref{lem:integral}:
    \formula{
    \lim_{\thet \to 0^+} \eta(\thet) &= \frac{2}{\pi} \int_0^\infty \frac{p_\alpha(t)}{1+t^2} \, dt = \frac{\pi(2-\alpha)}{4}.
    }
    Next, we study the derivative $\eta'(\thet)$. Direct differentiation gives
    \formula{
    \frac{\partial}{\partial \thet} \, \frac{\cos (\thet/2)}{\sqrt{t^2+\sin^2 (\thet/2)}} = - \frac{1}{2} \, \frac{(1+t^2) \sin (\thet/2)}{(t^2+ \sin^2 (\thet/2))^{3/2}}.
    }
    Differentiation under the integral sign in \eqref{eq:eta} is justified by dominated convergence. Therefore,
    \formula{
    -\eta'(\thet) = \frac{1}{\pi} \int_0^\infty \frac{t\sin (\thet/2)}{(t^2+\sin^2 (\thet/2))^{3/2}} \, p_\alpha(t) \, dt.
    }
    Note that $p_\alpha(t)$ is decreasing for $t\geq0$. Hence, $p_\alpha(t) \leq p_\alpha(0) = \arccot(-\cot (\alpha\pi/2)) = (2-\alpha)\pi/2$. It follows that
    \formula{
    -\eta'(\thet) \leq \frac{2-\alpha}{2} \int_0^\infty \frac{t\sin(\thet/2)}{(t^2+\sin^2(\thet/2))^{3/2}}\,dt = \frac{2-\alpha}{2}.
    }
    A standard integration by substitution shows that the above integral is equal to one.
\end{proof}

Lastly, we find the asymptotic expansions of $\eta(\thet)$ at $0$ and $\pi$. We formulate our results as two propositions below. 

\begin{proposition}
\label{prop:asymptotics:small:mu}
    For $0<\thet \leq \pi/2$, we have
    \formula{
    \eta(\thet) = \frac{(2-\alpha)\pi}{4} - \frac{2-\alpha}{2} \thet + E_0(\thet),
    }
    where
    \formula[eq:delta]{
    0<E_0(\thet) \leq C\begin{cases}
        (1-\alpha)^{-1} \thet^{\alpha+1}, &\text{for } \alpha \in (0,1), \\
        \thet^2 \log (e/\thet) \, &\text{for } \alpha =1, \\
        (\alpha-1)^{-1} \thet^2, &\text{for } \alpha \in (1,2).
    \end{cases}
    }
\end{proposition}

\begin{proof}
    Using \eqref{eq:eta} and Lemma \ref{lem:integral}, we find that
    \formula{
     \frac{(2-\alpha)\pi}{4}-\eta(\thet) = \frac{2}{\pi} \int_0^\infty \frac{1}{1+t^2} \, \biggl( 1- \frac{t\cos(\thet/2)}{\sqrt{t^2+\sin^2 (\thet/2)}} \biggr) \, p_\alpha(t) \, dt.
    }
    Let us denote
    \formula{
    f_\thet(t) &= \frac{1}{1+t^2} \, \biggl( 1- \frac{t\cos (\thet/2)}{\sqrt{t^2+\sin^2 (\thet/2)}} \biggr) = \frac{\sin^2 (\thet/2)}{\sqrt{t^2 + \sin^2 (\thet/2)}\bigl( t \cos (\thet/2) + \sqrt{t^2 + \sin^2 (\thet/2)} \bigr)}.
    }
    We have
    \formula{
    \int_0^\infty f_\thet(t) \, dt = \int_0^\infty \frac{1}{1+t^2} \, dt - \int_0^\infty \frac{t}{1+t^2} \, \frac{\cos(\thet/2)}{\sqrt{t^2+\sin^2 (\thet/2)}} \, dt = \frac{\pi}{2} - \biggl( \frac{\pi}{2} - \frac{\thet}{2}\biggr) = \frac{\thet}{2}.
    }
    It follows that
    \formula{
   \frac{(2-\alpha)\pi}{4}-\eta(\thet)  &= \frac{2}{\pi} \int_0^\infty f_\thet(t) p_\alpha(t) \, dt \\
    &= \frac{2}{\pi} \biggl( p_\alpha(0)\int_0^\infty f_\thet(t) \, dt + \int_0^\infty f_\thet(t) (p_\alpha(t)-p_\alpha(0)) \, dt \biggr) \\
    &= \frac{2-\alpha}{2}\thet - \frac{2}{\pi} \int_0^\infty f_\thet(t) (p_\alpha(0)-p_\alpha(t)) \, dt.
    }
    Therefore,
    \formula[eq:E0]{
    E_0(\thet) = \frac{2}{\pi} \int_0^\infty f_\thet(t) (p_\alpha(0)-p_\alpha(t)) \, dt.
    }
    Since $f_\theta(t)>0$ and $p_\alpha$ is strictly decreasing, $E_0(\theta)>0$.
    To prove the upper bound, we need the following two observations.

    First, for $0< \thet \leq \pi/2$, we have
    \formula{
    f_\thet(t) \leq \frac{\sin^2(\thet/2)}{t^2+\sin^2(\thet/2)} \leq \min\biggl\{1, \frac{\sin^2(\thet/2)}{t^2} \biggr\}.
    }
    
    Second, we claim that
    \formula[eq:qs]{
    p_\alpha(0)-p_\alpha(t) \leq C\min\{1, t^\alpha\}.
    }
    Note the identity $p_\alpha(0)-p_\alpha(t) = p_\alpha(t^{-1})$, which we have already exploited in the proof of Lemma~\ref{lem:integral}. If $t^\alpha \leq 1/2$, then we use $\arccot x \leq x^{-1}$ for $x>0$, and find that
    \formula{
    p_\alpha(t^{-1}) \leq \frac{t^\alpha \sin(\alpha\pi/2)}{1-t^\alpha\cos(\alpha\pi/2)} \leq 2t^\alpha.
    }
    If $t^\alpha > 1/2$, then we estimate $p_\alpha(t^{-1}) \leq \pi$, and \eqref{eq:qs} follows.

    With \eqref{eq:E0}--\eqref{eq:qs}, denoting $r=\sin(\thet/2)$, we obtain
    \formula[]{
    E_0(\thet) &\leq C\biggl( \int_0^r t^\alpha \, dt + r^2 \int_{r}^1 t^{\alpha-2} \, dt + r^2 \int_1^\infty t^{-2} \, dt \biggr) \nonumber \\
    &= C\biggl( \frac{r^{\alpha+1}}{\alpha+1} + r^2 \int_r^1 t^{\alpha-2} \, dt + r^2\biggr) \label{eq:E0:1}
    }
    If $\alpha \in (0,1)$, omitting the constant $C$, \eqref{eq:E0:1} becomes
    \formula{
    \frac{r^{\alpha+1}}{\alpha+1} + r^2 \frac{r^{\alpha-1}-1}{1-\alpha} + r^2 \leq \frac{r^{\alpha+1} +r^{\alpha+1}-r^2 +r^2}{1-\alpha} = \frac{2r^{\alpha+1}}{1-\alpha} \leq \frac{C\thet^{\alpha+1}}{1-\alpha}.
    }
    If $\alpha \in (1,2)$, we estimate the expression on the right-hand side of \eqref{eq:E0:1} as follows:
    \formula{
    \frac{r^{\alpha+1}}{\alpha+1} + r^2 \frac{1-r^{\alpha-1}}{\alpha-1} + r^2 \leq \frac{r^{\alpha+1} + r^2-r^{\alpha+1}+ r^2}{\alpha-1} = \frac{2r^2}{\alpha-1} \leq \frac{C\thet^{2}}{\alpha-1}.
    }
    If $\alpha =1$, we rewrite \eqref{eq:E0:1} as
    \formula{
    \frac{r^2}{2} + r^2 \log\frac{1}{r} + r^2 \leq Cr^2 \log\frac{e}{r} \leq C\thet^2 \log \frac{e\pi}{\thet} \leq C\thet^2 \log \frac{e}{\thet}.
    }
    In particular, for $\thet \in (0, \pi/2]$, we have $\log (e/\thet) \geq \log(2e/\pi) > 0$.
\end{proof}

\begin{proposition}
\label{prop:eta:big}
    For $\pi/4 \leq \thet < \pi$, we have
    \formula{
    \eta(\thet) = C_\alpha (\pi - \thet) + E_\pi(\thet),
    }
    where
    \formula{
    C_\alpha = \frac{1}{\pi} \int_0^\infty \frac{t}{(1+t^2)^{3/2}} \, p_\alpha(t) \,dt.
    }
    and
    \formula{
    0 < E_\pi(\thet) \leq C(\pi-\thet)^3.
    }
\end{proposition}

\begin{proof}
    Observe that
    \formula{
    \eta(\pi - \thet) = \frac{2}{\pi} \int_0^\infty \frac{t}{1+t^2} \, \frac{\sin (\thet/2)}{\sqrt{t^2 + \cos^2 (\thet/2)}} \, p_\alpha(t) \,dt.
    }
    Let $\smash{f_t(\thet) = \sin (\thet/2)(t^2+\cos^2(\thet/2))^{-1/2}}$. Since $f_t$ is odd, $f_t(0) = f_t''(0) = 0$. Moreover, $f'_t(0) = \smash{\frac{1}{2}(1+t^2)^{-1/2}}$. Applying Taylor's theorem yields
    \formula[eq:f:taylor]{
    f_t(\thet) = \frac{\thet}{2\sqrt{1+t^2}} + \frac{\thet^3}{6} f'''_t(\xi_t)
    }
    for $0<\thet \leq 3\pi/4$ and some $\xi_t \in (0, \thet)$.
    By direct calculation, $|f_t'''(\thet)|\leq C(1+t^2)^{-1/2}$ uniformly for $t\geq0$ and $0\leq \thet \leq3\pi/4$. Consequently, by \eqref{eq:f:taylor},
    \formula{
    |\eta(\pi - \thet) - C_\alpha \thet| &\leq \frac{2}{\pi} \int_0^\infty \frac{t}{1+t^2} \, p_\alpha(t) \, \bigg\lv f_t(\thet) - \frac{\thet}{2\sqrt{1+t^2}}\bigg\rv \, dt \\
    &\leq C\thet^3 \int_0^\infty \frac{t}{(1+t^2)^{3/2}} \, dt = C \thet^3,
    }
    which proves that $E_\pi(\thet) \leq C(\pi-\thet)^3$.
    
    It remains to show that $E_\pi(\thet)>0$. Let
    \formula{
    h_\thet(t) = f_t(\thet) - \frac{\thet}{2\sqrt{1+t^2}} = \frac{\thet}{2\sqrt{1+t^2}} \biggl( \frac{\sin(\thet/2)}{\thet/2} \, \sqrt{\frac{1+t^2}{t^2+\cos^2(\thet/2)}} -1 \biggr).
    }
    The quotient in the brackets above is a decreasing function of $t>0$. Furthermore, for $t=0$ it is equal to $\tan(\thet/2)/(\thet/2) > 1$, and for $t\to\infty$ it tends to $\sin(\thet/2)/(\thet/2) < 1$. Therefore, there is $t_0>0$ such that $h_\thet(t) > 0$ for $t<t_0$ and $h_\thet(t) < 0$ for $t>t_0$. Consider the following integral:
    \formula{
    \int_0^\infty \frac{t}{1+t^2} h_\thet(t) \, dt = \int_0^\infty \frac{t}{1+t^2} \, \frac{\sin(\thet/2)}{\sqrt{t^2+\cos^2(\thet/2)}}\, dt - \frac{\thet}{2}\int_0^\infty \frac{t}{(1+t^2)^{3/2}} \, dt.
    }
    The last integral above is equal to $1$. After substituting $s = \sqrt{t^2+\cos^2(\thet/2)}$, the middle integral is equal to
    \formula{
    \int_{\cos(\thet/2)}^\infty \frac{\sin(\thet/2)}{s^2+ \sin^2(\thet/2)} \, ds = \frac{\thet}{2}.
    }
    Therefore, 
    \formula{
    \int_0^\infty \frac{t}{1+t^2} h_\thet(t) \, dt =0.
    }
    Using this, we rewrite
    \formula{
    E_\pi(\pi-\thet) = \eta(\pi-\thet) - C_\alpha \thet &= \frac{2}{\pi} \int_0^\infty \frac{t}{1+t^2} h_\thet(t) p_\alpha(t)  \, dt \\ &= \frac{2}{\pi} \int_0^\infty \frac{t}{1+t^2} h_\thet(t) (p_\alpha(t)-p_\alpha(t_0)) \, dt >0,
    }
    where we used the monotonicity of $p_\alpha(t)$.
\end{proof}

\subsection{Difference quotients of the symbol}

The difference quotient of the symbol $\hat a$ is given by
\formula[eq:dq]{
\hat a (\thet;z)= \frac{(z-1)^2(\hat a(z)-g(\thet))}{(e^{i\thet}-z)(e^{-i\thet}-z)},
}
defined for $z \in \C \setminus ([0,\infty)\cup\{e^{i\thet}, e^{-i\thet}\})$ and extended continuously at $z = e^{\pm i \thet}$. Analogous difference quotient functions and their Wiener--Hopf factorisation played an important role in related works, e.g. $F(\lambda;x)$ in Eq. (1.8) of \cite{kw2026} or $b(t,s)$ on page 7 of \cite{bg2022}. The Wiener--Hopf factorisation of $\hat a (\thet;z)$ reads
\formula[eq:dq:wh]{
\hat a (\thet;z) = \hat a ^+(\thet;z)\hat a ^-(\thet;z^{-1}).
}
We normalise the factors by requiring $\hat a^+(\thet;0)=\hat a^-(\thet;0)>0$. With this normalisation, symmetry gives $\hat a^+(\thet;z)=\hat a^-(\thet;z)$. Thus, \eqref{eq:dq:wh} becomes
\formula{
\hat a(\thet;z) = \hat a^\pm(\thet;z) \hat a^\pm(\thet;z^{-1}).
}
Explicit formulas for $\hat a ^\pm(\thet;z)$ can be found in \cite{kw2026} (see Eq. (7.5) therein), although we do not need them here.

By (4.55)--(4.56) in \cite{kw2026}, the function $x\mapsto x/\hat a^\pm(\thet;1-x)$ is a complete Bernstein function. Moreover, by (4.57) therein, $(1-z)/\hat a^\pm(\thet;z)$ has a finite strictly positive limit as $z\to1$ nontangentially in $\mathbb C\setminus[1,\infty)$.

\subsection{Definition and properties of \texorpdfstring{$w_{\thet,k}$}{w-theta,k}}
\label{sec:wk}

The sequence $(w_{\thet,k}: k \geq 0)$ from \eqref{eq:v} is completely monotone and it is given by the following formula (see Eq. (7.6) in \cite{kw2026}):
\formula[eq:wk]{
w_{\thet,k} = \frac{1}{\pi} \, \sin \thet \, \frac{\lv \hat a^\pm(\thet; e^{i\thet}) \rv}{\lv 1 - e^{i\thet} \rv} \int_0^1 s^k \, \frac{\hat a^\pm(\thet; s)}{1-s} \, \frac{s^{\alpha/2} (1 - s)^{-\alpha} \sin(\alpha\pi/2)}{\lv e^{i \alpha\pi/2} - g(\thet) s^{\alpha/2} (1 - s)^{-\alpha} \rv^2} \, ds.
}
For convenience, we extend the sequence to all integer indices by setting $w_{\thet,k}=0$ for $k<0$.

Some results on $(w_{\thet,k})$ have already been provided in \cite{kw2026}. In particular, by Eqs. (1.15) and (1.16) therein,
\formula[eq:w:sin]{
0 \leq w_{\thet,k} \leq \sin\bigl(\thet + \tfrac{1}{2} \eta(\thet)\bigr)
}
and
\formula[eq:wk:sum:1]{
\sum_{k=0}^\infty w_{\thet,k} \leq \im \frac{e^{i\eta(\thet)/2}}{e^{-i\thet}-1} = \frac{\cos(\thet/2 + \eta(\thet)/2)}{2\sin(\thet/2)} \leq \frac{1}{2}\cot \frac{\thet}{2}.
}
Above we used $0 \leq \eta(\thet) \leq \pi-\thet$.

We now estimate the decay of $w_{\thet,k}$ as $k\to\infty$. We begin with two estimates for the Wiener--Hopf factors.

\begin{lemma}
\label{lem:F:estimate}
    For $x \in (0,1), \thet \in (0, \pi)$ we have
    \formula{
    \frac{\hat a^\pm(\thet;x)}{1-x} \leq \biggl(2\sin \frac{\thet}{2}\biggr)^{-(2-\alpha)/2}.
    } 
\end{lemma}

\begin{proof}
Recall that $x/\hat a^\pm(\thet;1-x)$ is a complete Bernstein function, hence it is increasing on $(0,1)$. It follows that $\hat a^\pm(\thet;x)/(1-x)$ is also increasing on $(0,1)$, and so
\formula{
\frac{\hat a^\pm(\thet;x)}{1-x} \leq \lim_{x \to 1^-} \frac{\hat a^\pm(\thet;x)}{1-x}.
}
Recall that by (4.57) in \cite{kw2026}, the above limit exists nontangentially in $\mathbb C\setminus[1,\infty)$ and belongs to $(0,\infty)$. For $|z|=1$, $z \neq 1$, we have
\formula[eq:F:prod]{
\frac{\hat a^\pm(\thet;z)}{1-z}\, \frac{\hat a^\pm(\thet;z^{-1})}{1-z^{-1}} = -\frac{z\hat a(\thet;z)}{(1-z)^2},
}
where $\hat a(\thet;z)$ is a difference quotient function given by \eqref{eq:dq}. Taking the limit as $z \to 1$ along the complex unit circle from above, we obtain
\formula[eq:limit]{
\lim_{z \to 1} \biggl( -\frac{z\hat a(\thet;z)}{(1-z)^2} \biggr) = -\frac{\hat a(1) - g(\thet)}{|e^{i\thet}-1|^2} = \frac{(2\sin \frac{\thet}{2})^\alpha}{(2\sin\frac{\thet}{2})^2} = \biggl(2\sin \frac{\thet}{2}\biggr)^{-(2-\alpha)}. 
}
By \eqref{eq:F:prod} and \eqref{eq:limit}, we have
\formula{
\lim_{x \to 1^-} \frac{\hat a^\pm(\thet;x)}{1-x} = \lim_{z \to 1} \frac{\hat a^\pm(\thet;z)}{1-z} = \biggl(2\sin \frac{\thet}{2}\biggr)^{-(2-\alpha)/2},
}
as desired.
\end{proof}

\begin{lemma}
\label{lem:F:estimate:2}
We have
\formula{
\frac{|\hat a^\pm(\thet;e^{i\thet})|}{|1-e^{i\thet}|} = \sqrt{\frac{\alpha}{2}} \biggl( 2\sin \frac{\thet}{2}\biggr)^{-(2-\alpha)/2}.
}
\end{lemma}

\begin{proof}
By the Wiener--Hopf factorisation, the identity $\hat a^\pm(\thet;e^{-it})=\overline{\hat a^\pm(\thet;e^{it})}$, and \eqref{eq:dq}, for $t\in(0,\pi)\setminus\{\thet\}$ we have
\formula{
\biggl|\frac{\hat a^\pm(\thet;e^{it})}{1-e^{it}}\biggr|^2=\frac{g(t)-g(\thet)}{2\cos\thet-2\cos t}.
}
Letting $t\to\thet$ and using continuity, we obtain
\formula{
\biggl|\frac{\hat a^\pm(\thet;e^{i\thet})}{1-e^{i\thet}}\biggr|^2=\frac{g'(\thet)}{2\sin\thet}=\frac{\alpha}{2}\biggl(2\sin\frac{\thet}{2}\biggr)^{\alpha-2}.
}
Taking square roots completes the proof.
\end{proof}

\begin{lemma}
\label{lem:last:int}
For $s \in (0,1)$ we have
\formula{
    \frac{s^{\alpha/2}(1-s)^{-\alpha}}{|e^{i\alpha\pi/2} - g(\thet) s^{\alpha/2} (1-s)^{-\alpha}|^2} \leq \frac{C}{\alpha^2} \, \frac{s^{-\alpha/2} (1-s)^\alpha}{ (g(\thet))^2}.
}
\end{lemma}

\begin{proof}
    Let $t = g(\thet) s^{\alpha/2}(1-s)^{-\alpha}$. Then $t>0$ and the assertion can be rewritten as
    \formula[eq:rewritten]{
    \frac{t^2}{|e^{i\alpha \pi/2} - t|^2} \leq \frac{C}{\alpha^2}.
    }
    For $\alpha \in [1,2)$, we estimate $|e^{i\alpha\pi/2}-t|^2 = t^2+1-2t\cos(\alpha\pi/2) \geq t^2$. Hence, \eqref{eq:rewritten} holds for some constant $C$. Now assume that $\alpha \in (0,1)$. Since $|e^{i\alpha \pi/2} - t| = |1 - e^{-i\alpha\pi/2}t| \geq |\im(1 - e^{-i\alpha\pi/2}t)| = t \sin (\alpha\pi/2)$, we estimate
    \formula*{
    \frac{t^2}{|e^{i\alpha \pi/2} - t|^2} \leq \frac{1}{\sin^2 (\alpha\pi/2)} \leq \frac{C}{\alpha^2}. \qedhere
    }  

\end{proof}

Now we are ready to obtain a decay estimate for the sequence $(w_{\thet,k})$. 

\begin{proposition}
\label{prop:wk}
    For every $k \geq 1$ we have
    \formula[eq:w:bound]{
    w_{\thet,k} \leq \frac{C}{\alpha^{3/2}}\,\frac{\cot(\thet/2)}{g(\thet)}\, \frac{1}{k^{\alpha+1}}.
    }
\end{proposition}

\begin{proof}
Using \eqref{eq:wk} and Lemmas \ref{lem:F:estimate}--\ref{lem:last:int}, we get
\formula{
w_{\thet,k} &\leq \frac{C}{\alpha^{3/2}} \biggl( \sin \frac{\thet}{2}\biggr)^{-\alpha-1} \cos \frac{\thet}{2} \int_0^1 s^{k-\alpha/2}(1-s)^{\alpha} \, ds \\
&= \frac{C}{\alpha^{3/2}} \, \frac{\cot (\thet/2)}{g(\thet)} \int_0^1 s^{k-\alpha/2}(1-s)^{\alpha} \, ds.
}
The integral above is the beta function. By the asymptotics of the ratio of gamma functions, that is $\Gamma(z)/\Gamma(z+\gamma) \leq Cz^{-\gamma}$ for $\gamma>0$ and $z\geq z_0>0$, we get
\formula{
\int_0^1 s^{k-\alpha/2}(1-s)^{\alpha} \, ds = \frac{\Gamma(k-\alpha/2+1)\Gamma(\alpha+1)}{\Gamma(k+\alpha/2+2)}\leq \frac{C}{k^{\alpha+1}},
}
and \eqref{eq:w:bound} follows.
\end{proof}

\subsection{Auxiliary estimates}
\label{sec:auxiliary:estimates}

\begin{lemma}
\label{lem:T:g}
    Assume that a sequence $(g_k: k \in \Z)$, such that $g_k = 0$ for $k<0$, satisfies
    \formula{
    & \sum_{k \in \Z} |g_k| \leq M_1, & \max_{0 \leq k \leq N/2}|g_k| \leq M_2, && \max_{-N/4\leq k \leq 3N/4} |\Delta^2 g_k| \leq M_3 &
    }
    for some $M_1,M_2,M_3>0$. Then, for $0 \leq k \leq N/2$ we have
    \formula{
    |(Tg)_k| \leq \frac{C}{N^{\alpha+1}} \biggl( \alpha M_1 + M_2 N + \frac{\alpha}{2-\alpha} M_3 N^3 \biggr).
    }
\end{lemma}

\begin{proof}
    Let $(g_k)$ be a sequence satisfying the assumptions. Since $a_0 = -\sum_{k \neq 0}a_k$, we get
    \formula[eq:T:lemma]{
    (Tg)_k = \sum_{j = -\infty}^\infty a_{j-k} g_j = \sum_{j \neq k} a_{j-k} g_j + a_0 g_k = \sum_{j \neq k} (g_j - g_k)a_{j-k}.
    }
    We claim that
    \formula[eq:claim:lemma]{
    |g_j - g_k - (j-k) \Delta g_k| \leq M_3(j-k)^2
    }
    whenever $-N/4\leq j,k \leq3N/4$.
    
    If $j=k$ or $j=k+1$, then \eqref{eq:claim:lemma} is obvious. Assume that $j \geq k+2$ and observe that
    \formula{
    g_j - g_k = \sum_{m=k}^{j-1} \Delta g_m &= \sum_{m=k}^{j-1} \biggl( \Delta g_k + \sum_{i=k}^{m-1} \Delta^2 g_i  \biggr) = (j-k) \Delta g_k + \sum_{m=k}^{j-1} \sum_{i=k}^{m-1} \Delta^2 g_i \\
    &= (j-k) \Delta g_k + \sum_{i=k}^{j-2} \sum_{m=i+1}^{j-1} \Delta^2 g_i = (j-k) \Delta g_k + \sum_{i=k}^{j-2} (j-i-1) \Delta^2 g_i.
    }
    Therefore,
    \formula{
    |g_j - g_k - (j-k) \Delta g_k| &= \bigg\lvert \sum_{i=k}^{j-2} (j-i-1) \Delta^2 g_i \bigg\rvert \\
    &\leq \max_{- N/4 \leq k \leq 3N/4} |\Delta^2 g_k| \bigg\lvert \sum_{i=k}^{j-2} (j-i-1) \bigg\rvert \leq M_3(j-k)^2.
    }
    Similarly, if $j<k$, we find that
    \formula{
    \lvert g_j - g_k - (j-k) \Delta g_k \rvert = \bigg\lvert \sum_{m=j}^{k-1} \sum_{i=m}^{k-1} 
    \Delta^2 g_i \bigg\rvert = \bigg\lvert \sum_{i=j}^{k-1} \sum_{m=j}^{i} 
    \Delta^2 g_i \bigg\rvert = \bigg\lvert \sum_{i=j}^{k-1} (i-j+1) \Delta^2 g_i \bigg\rvert \leq M_3 (j-k)^2,
    }
    which proves our claim.
    
    Let $0 \leq k \leq N/2$. We denote $J_1 = \{j \neq k: |j-k|\leq N/4\}$ and $J_2 = \{j: |j-k| > N/4\}$. By symmetry, $\sum_{j\in J_1}(j-k)a_{j-k}=0$. The right-hand side of \eqref{eq:T:lemma} is therefore equal to
    \formula{
    \sum_{j \in J_1} (g_j - g_k - (j-k) \Delta g_k) a_{j-k} + \sum_{j \in J_2} (g_j - g_k) a_{j-k}.
    }
    Note that $0 \leq k \leq N/2$ and $|j-k| \leq N/4$ imply that $-N/4 \leq j \leq 3N/4$. Hence, using Lemma~\ref{lem:ak:bounds} and \eqref{eq:claim:lemma}, we estimate the first sum:
    \formula[]{
    \bigg\lvert \sum_{j \in J_1} (g_j - g_k - (j-k) \Delta g_k) a_{j-k} \bigg\rvert &\leq M_3 \sum_{j \in J_1} (j-k)^2 \lv a_{j-k} \rv \leq C \alpha M_3 \sum_{j \in J_1} \frac{1}{|j-k|^{\alpha-1}} \nonumber \\
    &= C \alpha M_3 \sum_{j=1}^{\floor{N/4}} \frac{1}{j^{\alpha-1}}\leq \frac{C \alpha M_3}{2-\alpha} \frac{1}{N^{\alpha-2}}. \label{eq:lem:ineq:1}
    }
The remaining sum reads
\formula[]{
\bigg\lvert \sum_{j \in J_2} (g_j - g_k) a_{j-k} \bigg\rvert &\leq \lv a_{\floor{N/4}+1} \rv \sum_{j \in J_2} |g_j| + |g_k| \sum_{j \in J_2} \lv a_{j-k} \rv \nonumber \\
&\leq \frac{C \alpha }{N^{\alpha+1}}  \sum_{k \in \Z} |g_k| + \max_{0 \leq k \leq N/2} |g_k| \sum_{j \geq N/4} \frac{C \alpha }{j^{\alpha+1}} \leq \frac{C \alpha  M_1}{N^{\alpha+1}} + \frac{C M_2}{N^{\alpha}}. \label{eq:lem:ineq:2}
}
Using \eqref{eq:T:lemma}, \eqref{eq:lem:ineq:1} and \eqref{eq:lem:ineq:2}, we find that for $0 \leq k\leq N/2$
\formula{
|(Tg)_k| \leq \frac{C \alpha  M_1}{N^{\alpha+1}} + \frac{C M_2}{N^{\alpha}} +\frac{C \alpha M_3}{2-\alpha} \frac{1}{N^{\alpha-2}} = \frac{C}{N^{\alpha+1}} \biggl( \alpha M_1 + M_2 N + \frac{\alpha}{2-\alpha} M_3 N^3 \biggr),
}
as desired.
\end{proof}

\begin{lemma}
\label{lem:T:hk}
Let $\ph\in(0,2\pi)$, $\psi\in\R$, and $h_j=\sin((j+1)\ph+\psi)$. For $0\leq k\leq N/2$, we have
\formula{
\biggl|\sum_{j=N}^\infty a_{j-k}h_j\biggr|\leq\frac{C \alpha }{N^{\alpha+1}\sin^2(\ph/2)}\biggl(\sin\frac{\ph}{2}\,\bigl|\cos((N+\tfrac12)\ph+\psi)\bigr|+|\sin\ph|+\frac{|\sin\psi|}{N}\biggr).
}
\end{lemma}

\begin{proof}
Let $M \geq N$ and observe that
\formula{
H_M = \sum_{j=N}^M h_j = \frac{\cos\bigl(\ph(N+1/2)+\psi\bigr)-\cos\bigl(\ph(M+3/2) +\psi\bigr)}{2\sin (\ph/2)}.
}
Using Abel's summation by parts, we find that
\formula{
\sum_{j=N}^M a_{j-k} h_j = a_{M-k} H_M  - \sum_{j=N}^{M-1} H_j \Delta a_{j-k}.
}
Since $H_M$ is bounded and $a_{M-k}\to0$ as $M\to\infty$, letting $M\to\infty$ yields
\formula[]{
 \sum_{j=N}^\infty a_{j-k} h_j &= -\sum_{j=N}^\infty H_j \Delta a_{j-k} \nonumber \\
&=\frac{1}{2\sin (\ph/2)} \biggl( -\cos\bigl(\ph(N+\tfrac{1}{2})+\psi\bigr) \sum_{j=N}^\infty \Delta a_{j-k} +\sum_{j=N}^\infty \cos\bigl(\ph(j+\tfrac{3}{2})+\psi\bigr)\Delta a_{j-k}\biggr) \label{eq:sum:1}
}
By Lemma~\ref{lem:ak:bounds}, we have
\formula[eq:sum:2]{
 \sum_{j=N}^\infty \Delta a_{j-k} = -a_{N-k} \leq \frac{C \alpha }{N^{\alpha+1}}.
}
Applying Abel's summation formula once more, we obtain
\formula*[eq:sum:3]{
\sum_{j=N}^\infty \cos&\bigl(\ph(j+\tfrac{3}{2})+\psi\bigr)\Delta a_{j-k} \\
&= \frac{1}{2\sin(\ph/2)} \biggl( - \sum_{j=N}^\infty \sin\bigl(\ph(j+2)+\psi \bigr) \Delta^2 a_{j-k} - \sin\bigl((N+1)\ph+\psi\bigr) \Delta a_{N-k}\biggr).
}
Note that for $m \geq 0$ we have
\formula[eq:sin:sum]{
\lvert \sin(\ph m +\psi)\rvert \leq \lvert \sin(\ph m)\rvert + \lvert\sin \psi\rvert \leq m \lvert \sin \ph\rvert + \lvert\sin \psi\rvert.
}
Therefore,
\formula[eq:sum:4]{
\bigg\lvert \sum_{j=N}^\infty \sin\bigl(\ph(j+2)+\psi \bigr) \Delta^2 a_{j-k} \bigg\rvert \leq \lvert \sin \ph \rvert \sum_{j=N}^\infty (j+2) |\Delta^2 a_{j-k}| + \lvert \sin \psi\rvert \Delta a_{N-k}.
}
By Lemma~\ref{lem:ak:bounds},
\formula[eq:Dak:lemma]{
\Delta a_{N-k} \leq \frac{C \alpha }{N^{\alpha+2}}
}
and
\formula{
\sum_{j=N}^\infty (j+2) |\Delta^2 a_{j-k}| \leq \sum_{j=N}^\infty \frac{C \alpha }{j^{\alpha+2}} \leq \frac{C \alpha }{N^{\alpha+1}}.
}
Thus, we can estimate \eqref{eq:sum:4} as follows:
\formula[eq:new:1]{
\bigg\lvert \sum_{j=N}^\infty \sin\bigl(\ph(j+2)+\psi \bigr) \Delta^2 a_{j-k} \bigg\rvert \leq \frac{C \alpha }{N^{\alpha+1}} \biggl(\lvert\sin \ph\rvert + \tfrac{1}{N}\lvert\sin \psi\rvert \biggr).
}
Again, using \eqref{eq:sin:sum} and \eqref{eq:Dak:lemma}, we get
\formula[eq:new:2]{
\big\lvert \sin\bigl((N+1)\ph + \psi\bigr) \Delta a_{N-k} \big\rvert \leq \frac{C \alpha }{N^{\alpha+1}} \biggl(\lvert\sin \ph\rvert + \tfrac{1}{N}\lvert\sin \psi\rvert \biggr)
}
By \eqref{eq:new:1} and \eqref{eq:new:2}, we estimate \eqref{eq:sum:3}:
\formula[eq:sum:5]{
\bigg\lvert \sum_{j=N}^\infty \cos\bigl(\ph(j+\tfrac{3}{2})+\psi\bigr)\Delta a_{j-k} \bigg\rvert \leq \frac{C \alpha }{N^{\alpha+1} \sin (\ph/2)}\, \biggl(\lvert\sin \ph\rvert + \tfrac{1}{N}\lvert\sin \psi\rvert \biggr).
}
Finally, by \eqref{eq:sum:1}, \eqref{eq:sum:2}, \eqref{eq:sum:5}, we arrive at the desired estimate.
\end{proof}

\section{Approximate eigenvectors}
\label{sec:approximate:eigenvectors}

\subsection{Construction of the approximate eigenvectors}
\label{sec:construction}

Let $N \geq 2$. We construct the approximations $\tilde v_{n,k}$ of the eigenvectors of $T_N$ by combining eigenvectors of $T$ and connecting them in a sufficiently smooth way with an appropriate sequence $(q_k: k \in \Z)$ which will be defined later in Section~\ref{sec:proofs}. Our methods were inspired by \cite{kkms2010, kwasnicki2012, kkm2013, kkm2016}, but needed to be adapted to the discrete setting.

The approximate eigenvectors are defined in the following way:
\formula[eq:v:approx]{
\tilde{v}_{n,k} = v_{\thet_n, k} q_{N-1-k} - (-1)^{n} v_{\thet_n, N-1-k}q_k.
}
for $1\leq n\leq N$ and $0\leq k\leq N-1$.

We choose $\thet_n$ so that the sine components of $v_{\thet_n,k}$ and $(-1)^{n-1}v_{\thet_n,N-1-k}$ agree. More precisely, we require:
\formula[eq:sines]{
\sin\biggl((k+1)\thet_n+\frac{\eta(\thet_n)}{2}\biggr)=(-1)^{n-1}\sin\biggl((N-k)\thet_n+\frac{\eta(\thet_n)}{2}\biggr).
}
Observe that \eqref{eq:sines} is satisfied if
\formula[eq:arg:eq]{
(N+1) \thet + \eta(\thet) = n \pi,
}
where $1 \leq n \leq N$.

From now on, we refer to equation \eqref{eq:arg:eq} as the \textit{phase equation}.

\subsection{Solutions of the phase equation}

In this section, we will prove the existence and uniqueness of the solutions of the phase equation \eqref{eq:arg:eq} and study their distribution.

\begin{proposition}
\label{prop:solutions}
    For every $1 \leq n \leq N$, the phase equation \eqref{eq:arg:eq} has exactly one solution $\thet_n \in (0,\pi)$.
\end{proposition}

\begin{proof}
    Let us denote the left-hand side of \eqref{eq:arg:eq} by
    \formula{
    \zeta(\thet) = (N+1)\thet + \eta(\thet).
    }
    Note that this is a continuous function of $\thet$. By Proposition~\ref{lem:theta:plus},
    \formula{
    & \lim_{\thet \to 0^+} \zeta(\thet) = \frac{\pi(2-\alpha)}{4} < \pi, &
    \lim_{\thet \to \pi^-} \zeta(\thet) = (N+1)\pi. &
    }
    Furthermore, 
    \formula{
    \zeta'(\thet)=N+1+\eta'(\thet)\geq N+\frac{\alpha}{2}>0.
    }
    The intermediate value theorem therefore gives a solution for every $1\leq n\leq N$.
\end{proof}

By $\thet_n$ we denote the unique solution of the phase equation~\eqref{eq:arg:eq}. Since the left-hand side of \eqref{eq:arg:eq} is an increasing function of $\thet$, we have
\formula{
0<\thet_1 < \thet_2 < \ldots < \thet_N < \pi.
}
Our next result provides bounds for $\thet_n$.

\begin{lemma}
\label{lem:theta:bounds}
If $1 \leq n \leq N$, then
\formula[eq:thet:bounds]{
\max\biggl\{\frac{(n-1)\pi}{N}, \frac{(2n-1)\pi}{2(N+1)}\biggr\} \leq \thet_n \leq \frac{n\pi}{N+1}.
}
In particular, $\thet_n \approx n/N$ and $\pi - \thet_n \approx 1-n/(N+1)$.
\end{lemma}

\begin{proof}
We apply \eqref{eq:thetas:inequality} and \eqref{eq:theta:plus:bounds} to the phase equation. More precisely, by \eqref{eq:thetas:inequality}, we find that
\formula{
n\pi = (N+1) \thet_n + \eta(\thet_n) \leq N \thet_n + \pi,
}
hence $\thet_n \geq (n-1)\pi/N$. Similarly, using \eqref{eq:theta:plus:bounds}, we obtain
\formula{
(N+1)\thet_n \leq n\pi \leq (N+1) \thet_n + \tfrac{\pi}{2},
}
and the remaining inequalities follow. With \eqref{eq:thet:bounds}, one can verify that
\formula{
& \frac{\pi}{4} \, \frac{n}{N} \leq \thet_n \leq \pi \frac{n}{N} && \text{and} & \pi \biggl( 1-\frac{n}{N+1} \biggr) \leq \pi-\thet_n \leq \frac{3\pi}{2} \biggl( 1- \frac{n}{N+1} \biggr), &
}
which proves the last part of the lemma.
\end{proof}

We next collect some useful estimates.

\begin{lemma}
\label{lem:asymptotics}
    Let $\xi_n \in (0,\pi)$ satisfy $\xi_n \approx n/N$ and $\pi - \xi_n \approx 1-n/(N+1)$. Then
    \formula[eq:sin:cos]{
    & \sin \frac{\xi_n}{2} \approx \frac{n}{N}, & \cos\frac{\xi_n}{2} \approx 1-\frac{n}{N+1}. &
    }
    We also have
    \formula{
    0  < \sin \frac{\eta(\xi_n)}{2} \leq C\biggl( 1-\frac{n}{N+1} \biggr).
    }
    Moreover, for $g$ defined in \eqref{eq:G},
    \formula[eq:G']{
    g'( \xi_n) \approx  \alpha \biggl( 1 - \frac{n}{N+1} \biggr) \biggl( \frac{n}{N}\biggr)^{\alpha-1}.
    }
\end{lemma}

\begin{proof}
    Properties \eqref{eq:sin:cos} follow from $\sin x \approx x$ for $x \in [0,\pi/2]$ and $\sin(\pi/2-x)=\cos x$. By \eqref{eq:thetas:inequality} and \eqref{eq:sin:cos}, we find that
    \formula{
    0 < \sin \frac{\eta(\xi_n)}{2} \leq \cos \frac{\xi_n}{2} \leq C\biggl(1-\frac{n}{N+1}\biggr).
    }
    Finally, we observe that $\smash{g'(\thet) = \alpha \cos (\thet/2) (2\sin (\thet/2))^{\alpha-1}}$ and use \eqref{eq:sin:cos} to get \eqref{eq:G'}.
\end{proof}

We will be particularly interested in the behaviour of the sequence $(w_{\thet_n,k}: k \geq 0)$ that appears in our approximate eigenvectors $\tilde v_n$ constructed in Section \ref{sec:construction}. Using Lemmas~\ref{lem:theta:bounds} and \ref{lem:asymptotics} combined with Proposition \ref{prop:wk}, we deduce that
\formula[eq:w:theta:n]{
|w_{\thet_n, k}| \leq\frac{C}{\alpha^{3/2}}\frac{1}{k^{\alpha+1}} \biggl( \frac{N}{n} \biggr)^{\alpha+1} \biggl( 1- \frac{n}{N+1} \biggr)
}
whenever $k \geq 1$. By \eqref{eq:wk}, we also have for $m \geq 0$,
\formula*[eq:Dw:theta:n]{
|\Delta^m w_{\thet_n, k}| &\leq \frac{C}{\alpha^{3/2}} \, \frac{\cot(\thet_n/2)}{g(\thet_n)} \int_0^1 s^{k-\alpha/2}(1-s)^{\alpha+m} \, ds \\
&\leq 
\frac{C_m}{\alpha^{3/2}} \, \frac{1}{k^{\alpha+m+1}} \biggl( \frac{N}{n} \biggr)^{\alpha+1} \biggl( 1- \frac{n}{N+1} \biggr).
}
Moreover, \eqref{eq:wk:sum:1} implies that
\formula[eq:w:sum:theta:n]{
\sum_{k=0}^\infty w_{\thet_n,k} \leq \frac{C N}{n} \biggl(1-\frac{n}{N+1} \biggr).
}

\section{Proof of the main results}
\label{sec:proofs}

\subsection{Asymptotics of eigenvalues}
\label{sec:proof:thm2}

Let $N \geq 2$. Recall that in \eqref{eq:v:approx} we defined the approximate eigenvectors for $0 \leq k \leq N-1$:
\formula{
\tilde{v}_{n,k} = v_{\thet_n, k} q_{N-1-k} - (-1)^{n} v_{\thet_n, N-1-k}q_k,
}
where $\thet_n$ are the solutions of the phase equation \eqref{eq:arg:eq} and $(q_k)$ is a suitable sequence. We extend $v_{\thet_n,k}$ by zero for $k<0$ and use the same formula to define $\tilde v_{n,k}$ for all $k\in\Z$. As in \cite{kwasnicki2012}, one can take the following quadratic sequence: 
\formula{
q_k = \begin{cases}
    0 &\text{for } k < (N-1)/4, \\
    \frac{8}{(N-1)^2}\bigl(k- \tfrac{N-1}{4}\bigr)^2 &\text{for } (N-1)/4 \leq k < \frac{N-1}{2}, \\
    1-\frac{8}{(N-1)^2}\bigl(k- \tfrac{3(N-1)}{4}\bigr)^2 &\text{for } \frac{N-1}{2} \leq k < \frac{3(N-1)}{4}, \\
    1 &\text{for } k \geq \frac{3(N-1)}{4}.
\end{cases}
}
Clearly, $0 \leq q_k \leq 1$ and $q_k + q_{N-1-k}=1$ for every $k \in \Z$. It can be verified that
\formula[eq:Dq]{
& |\Delta q_k| \leq \frac{C}{N}, & |\Delta^2 q_k| \leq \frac{C}{N^2}. &
}
Since $q_k + q_{N-1-k}=1$, we have
\formula*[eq:v:vn]{
    \tilde{v}_{n,k} - v_{\thet_n, k} &= v_{\thet_n,k} (q_{N-1-k} - 1) - (-1)^{n} v_{\thet_n, N-1-k} q_k \\
    &= -v_{\thet_n, k} q_k - (-1)^{n} v_{\thet_n, N-1-k} q_k \\
    &= - q_k ( v_{\thet_n, k} + (-1)^{n} v_{\thet_n, N-1-k} ) \\
    &= q_k (w_{\thet_n, k} + (-1)^{n} w_{\thet_n, N-1-k} ) - h_k,
}
where we denoted
\formula{
    h_k = q_k \bigl( \sin((k+1) \thet_n + \tfrac{1}{2}\eta(\thet_n)) \ind_{k \geq 0} + (-1)^{n} \sin\bigl((N-k) \thet_n + \tfrac{1}{2}\eta(\thet_n)\bigr) \ind_{k \leq N-1}\bigr).
}
If $k \leq -1$, then $q_k = 0$, hence $h_k=0$. If $0 \leq k \leq N-1$, we also have $h_k = 0$ because of \eqref{eq:sines}. If $k \geq N$, we have $q_k=1$ and $\ind_{k \leq N-1} = 0$. It follows that
\formula{
    h_k = \sin\bigl((k+1) \thet_n + \tfrac{1}{2}\eta(\thet_n)\bigr) \ind_{k \geq N}.
}
Consequently,
    \formula[eq:difference:v]{
    \tilde{v}_{n,k} - v_{\thet_n, k} = q_k \bigl(w_{\thet_n, k} + (-1)^{n} w_{\thet_n, N-1-k} \bigr) -\sin\bigl((k+1) \thet_n + \tfrac{1}{2}\eta(\thet_n)\bigr) \ind_{k \geq N}.
}
If we additionally denote
\formula{
    f_k &= w_{\thet_n, k} + (-1)^{n} w_{\thet_n, N-1-k},\\
    g_k &= f_k q_k,
}
then \eqref{eq:difference:v} reads
\formula[eq:v:tilde]{
    \tilde{v}_{n,k} = v_{\thet_n, k} + g_k - h_k.
}

Note that $\tilde{v}_{n,k} = 0$ for $k < 0$ and $k \geq N$. Because our analysis involves both finite Toeplitz matrices $T_N$ and infinite Toeplitz matrices $T$, the object $\tilde{v}_n$ will be treated interchangeably as a vector in $\mathbb{R}^N$ or as an element of $\mathbb{R}^{\{0,1,2,\ldots\}}$, depending on the operator applied to it.

\begin{lemma}
\label{lem:TN}
    For $1 \leq n \leq N$ we have
    \formula{
    \lV T_N \tilde{v}_n - g(\thet_n) \tilde v_n \rV \leq \frac{C}{(2-\alpha)\alpha^{3/2}} \, \frac{1}{nN^{\alpha-1/2}} \biggl(1-\frac{n}{N+1}\biggr).
    }
\end{lemma}

\begin{proof}
    By \eqref{eq:v:tilde}, we find that
    \formula{
    |(T \tilde{v}_{n})_k - g(\thet_n) \tilde v_{n,k}| \leq |(T v_{\thet_n})_k - g(\thet_n) v_{\thet_n,k}| + |(T g)_k| + |(T h)_k| + g(\thet_n) (|g_k| + |h_k|).
    }
    Since $Tv_{\thet_n}=g(\thet_n)v_{\thet_n}$, the first term vanishes. Furthermore, $h_k$ vanishes for $k<N$. We assume that $k \leq (N-1)/2$ and inspect the remaining terms one by one. For this purpose, we divide the rest of the proof into four steps.

\textit{Step 1.} We estimate $|(Tg)_k|$ by verifying the assumptions of Lemma~\ref{lem:T:g}. Set
\formula{
A=\frac{1}{\alpha^{3/2}n^{\alpha+1}}\biggl(1-\frac{n}{N+1}\biggr).
}
Since $q_k=0$ for $k<0$, we also have $g_k=0$ there. By the definition of $(q_k)$ and \eqref{eq:w:sum:theta:n},
\formula{
\sum_{k\in\Z}|g_k|\leq\sum_{k\geq(N-1)/4}\bigl(w_{\thet_n,k}+w_{\thet_n,N-1-k}\bigr)\leq2\sum_{k=0}^\infty w_{\thet_n,k}\leq\frac{CN}{n}\biggl(1-\frac{n}{N+1}\biggr).
}

For $N\geq3$, monotonicity and \eqref{eq:w:theta:n} give
\formula[eq:f:bounds]{
\max_{0\leq k\leq N/2}|g_k|\leq w_{\thet_n,\lceil(N-1)/4\rceil}+w_{\thet_n,\lceil N/2\rceil-1}\leq CA.
}
For $N=2$, we have $g_0=0$ and $g_1=f_1$, so the same bound follows from
\formula{
|g_1|\leq w_{\thet_n,1}+w_{\thet_n,0}\leq\frac{C}{n}\biggl(1-\frac{n}{N+1}\biggr)\leq CA,
}
where we used \eqref{eq:w:sum:theta:n} and the boundedness of $\alpha^{3/2}n^\alpha$ for $n\leq2$.

We next estimate the second differences. The discrete product rule yields
\formula[eq:Dg2]{
|\Delta^2g_k|\leq|f_{k+2}|\,|\Delta^2q_k|+|\Delta q_k|\bigl(|\Delta f_{k+1}|+|\Delta f_k|\bigr)+|q_{k+1}|\,|\Delta^2f_k|.
}
First assume that $N\geq32$. If $k<(N-1)/4-2$, then $\Delta^2q_k=\Delta q_k=q_{k+1}=0$, and hence $\Delta^2g_k=0$. Hence, we can consider $(N-1)/4-2\leq k\leq3N/4$. For $m=0,1,2$, we have
\formula{
\Delta^m f_k=\Delta^m w_{\thet_n,k}+(-1)^{n+m}\Delta^m w_{\thet_n,N-k-m-1}.
}
By \eqref{eq:Dw:theta:n}, we obtain
\formula{
|\Delta^m f_k|\leq C_m AN^{\alpha+1}\biggl(\frac{1}{k^{\alpha+m+1}}+\frac{1}{(N-k-m-1)^{\alpha+m+1}}\biggr)
}
for $1 \leq k \leq N-m-2$. The smallest indices needed when applying this estimate to the terms in \eqref{eq:Dg2} are $k$ and $N-k-3$. Under our assumptions,
\formula{
& k\geq\frac{N-1}{4}-2\geq\frac N8, & N-k-3\geq\frac N4-3\geq\frac N8. &
}
Consequently,
\formula{
|f_{k+2}|\leq CA,\qquad |\Delta f_k|+|\Delta f_{k+1}|\leq\frac{CA}{N},\qquad |\Delta^2f_k|\leq\frac{CA}{N^2}.
}
Together with \eqref{eq:Dg2} and \eqref{eq:Dq}, this implies that
\formula{
\max_{-N/4\leq k\leq3N/4}|\Delta^2g_k|\leq\frac{CA}{N^2}.
}
For $2\leq N<32$, the same bound follows from
\formula{
\sup_{k\in\Z}|\Delta^2g_k|\leq4\sup_{k\in\Z}|f_k|\leq8\sum_{k\geq0}w_{\thet_n,k}\leq\frac{CN}{n}\biggl(1-\frac{n}{N+1}\biggr)\leq\frac{CA}{N^2},
}
since $\alpha^{3/2}n^\alpha N^3$ is bounded in this range.

    Thus, Lemma~\ref{lem:T:g} applies with
    \formula{
    & M_1=\frac{CN}{n}\biggl(1-\frac{n}{N+1}\biggr), & M_2=CA, && M_3=\frac{CA}{N^2}. &
    }
    For $0\leq k\leq(N-1)/2$, it follows that
    \formula{
    |(Tg)_k|\leq\frac{C}{N^{\alpha+1}}\biggl(\alpha M_1+\frac{AN}{2-\alpha}\biggr)\leq\frac{C}{nN^\alpha}\biggl(\alpha+\frac{1}{(2-\alpha)\alpha^{3/2}n^\alpha}\biggr)\biggl(1-\frac{n}{N+1}\biggr).
    }

    \textit{Step 2.} To estimate $|(Th)_k|$, we apply Lemma \ref{lem:T:hk} with $\ph = \thet_n$ and $\psi = \eta(\thet_n)/2$:
    \formula{
    \lv (Th)_k \rv \leq \frac{C \alpha }{N^{\alpha+1}\sin^2(\thet_n/2)}\biggl(\sin\frac{\thet_n}{2}\,\bigl|\cos\bigl((N+\tfrac12)\thet_n+\tfrac{1}{2} \eta(\thet_n)\bigr)\bigr|+|\sin\thet_n|+\frac{|\sin(\eta(\thet_n)/2)|}{N}\biggr).
    }
    By the phase equation \eqref{eq:arg:eq} and $0\leq\eta(\thet_n)\leq\pi-\thet_n$, we have
    \formula{
    \bigl\lv \cos\bigl((N+\tfrac12)\thet_n+\tfrac12\eta(\thet_n)\bigr)\bigr\rv = \cos\frac{\thet_n+\eta(\thet_n)}{2}\leq\cos\frac{\thet_n}{2}.
    }
    Using Lemma~\ref{lem:asymptotics} and the identity $\sin\thet_n = 2\sin(\thet_n/2) \cos(\thet_n/2)$, we obtain
    \formula{
    \lv (Th)_k \rv \leq \frac{C \alpha }{N^{\alpha+1}}\frac{(n/N)(1-n/(N+1))}{n^2/N^2} = \frac{C \alpha }{nN^\alpha}\biggl(1-\frac{n}{N+1}\biggr).
    }

    \textit{Step 3.} Now we will estimate the remaining term, which is $g(\thet_n)|g_k|$. If $k < (N-1)/4$, we have $g_k = 0$. If $(N-1)/4 \leq k \leq (N-1)/2$, then by \eqref{eq:f:bounds}, the definition \eqref{eq:G} of $g$, and Lemma~\ref{lem:theta:bounds}, we have
    \formula{
    g(\thet_n) \lv g_k \rv \leq g(\thet_n) \max_{0 \leq k \leq N/2} \lv g_k \rv \leq CA\biggl(\frac nN\biggr)^\alpha = \frac{C}{\alpha^{3/2}}\,\frac{1}{nN^\alpha}\biggl(1-\frac{n}{N+1}\biggr).
    }
    
    \textit{Step 4.} Combining the previous estimates and using $n^{-\alpha}\leq1$, we obtain
    \formula{
    |(T\tilde v_n)_k-g(\thet_n)\tilde v_{n,k}|\leq\frac{C}{(2-\alpha)\alpha^{3/2}}\,\frac{1}{nN^\alpha}\biggl(1-\frac{n}{N+1}\biggr)
    }
    for $0\leq k\leq(N-1)/2$. Since $\tilde v_n$ vanishes outside $\{0,\ldots,N-1\}$, we have $(T\tilde v_n)_k=(T_N\tilde v_n)_k$ for $0\leq k\leq N-1$. Since $a_j=a_{-j}$, the matrix $T_N$ commutes with the reflection $k\mapsto N-1-k$, and the estimate holds for every $0\leq k\leq N-1$. Consequently,
    \formula{
    \lV T_N \tilde{v}_{n} - g(\thet_n) \tilde v_{n} \rV^2 &\leq \sum_{k=0}^{N-1} \frac{C}{(2-\alpha)^2\alpha^3} \, \frac{1}{n^2N^{2\alpha}} \biggl( 1-\frac{n}{N+1}\biggr)^2 \\
    &= \frac{C}{(2-\alpha)^2\alpha^3} \, \frac{1}{n^2N^{2\alpha-1}} \biggl( 1-\frac{n}{N+1}\biggr)^2.
    }
    Taking square roots completes the proof.
\end{proof}

\begin{lemma}
\label{lem:v:norm}
    For $1 \leq n \leq N$ we have
    \formula{
    \bigg\lvert\lVert \tilde v_n \rVert^2 - \frac{N}{2} \bigg\rvert \leq \frac{C N}{n}. 
    }
    In particular, there is a constant $C>0$ such that $\lVert \tilde v_n \rVert^2 > N/4$ for all $C \leq n \leq N$.
\end{lemma}

\begin{proof}
    Let
    \formula{
    &\ph_k = (k+1)\thet_n + \tfrac{1}{2}\eta(\thet_n), & r_k = w_{\thet_n,k} q_{N-1-k} - (-1)^n w_{\thet_n, N-1-k} q_k.&
    }
    Similarly to \eqref{eq:v:vn}, one can show that for $0 \leq k \leq N-1$,
    \formula{
    \tilde v_{n,k} = \sin \ph_k - r_k.
    }
    Observe that
    \formula{
    \bigg\lvert\lVert \tilde v_n \rVert^2 - \frac{N}{2} \bigg\rvert &= \biggl\lvert \sum_{k=0}^{N-1} (\tilde v_{n,k}^2  - \tfrac{1}{2}) \biggr\rvert \\ &\leq \biggl\lvert \sum_{k=0}^{N-1} (\sin^2 \ph_k - \tfrac{1}{2}) \biggr\rvert + 2 \biggl\lvert \sum_{k=0}^{N-1} r_k \sin \ph_k \biggr\rvert + \sum_{k=0}^{N-1} r_k^2.
    }
    The first component above gives
    \formula{
    \biggl\lvert \sum_{k=0}^{N-1} (\sin^2 \ph_k - \tfrac{1}{2}) \biggr\rvert &= \frac{1}{2} \biggl\lvert \sum_{k=0}^{N-1} \cos (2\ph_k) \biggr\rvert = \frac{1}{2} \biggl\lvert \sum_{k=0}^{N-1} \cos (2(k+1)\thet_n + \eta(\thet_n)) \biggr\rvert \\
    &= \frac{1}{2} \biggl\lvert \frac{\sin(N\thet_n) \cos((N+1) \thet_n + \eta(\thet_n))}{\sin \thet_n} \biggr\rvert \leq \frac{\lvert\sin (N\thet_n)\rvert}{4\sin (\thet_n/2)\cos (\thet_n/2)} .
    }
    By the phase equation and \eqref{eq:thetas:inequality}, we get 
    \formula{
    \lvert\sin(N\thet_n)\rvert &= \lvert\sin(n\pi-\thet_n-\eta(\thet_n))\rvert = \sin(\thet_n + \eta(\thet_n)) \\ &= 2 \sin \biggl( \frac{\thet_n}{2} + \frac{\eta(\thet_n)}{2}\biggr) \cos \biggl( \frac{\thet_n}{2} + \frac{\eta(\thet_n)}{2}\biggr) \leq 2 \cos \frac{\thet_n}{2}.
    }
    Using Lemma \ref{lem:asymptotics}, we find that
    \formula{
    \biggl\lvert \sum_{k=0}^{N-1} (\sin^2 \ph_k - \tfrac{1}{2}) \biggr\rvert \leq \frac{1}{2 \sin (\thet_n/2)} \leq \frac{C N}{n}.
    }   
    Next, observe that \eqref{eq:w:sin} implies that $0 \leq w_{\thet_n,k} \leq 1$. Hence, $|r_k| \leq 1$. Using this and \eqref{eq:w:sum:theta:n}, we obtain
    \formula{
    2 \biggl\lvert \sum_{k=0}^{N-1} r_k \sin \ph_k \biggr\rvert + \sum_{k=0}^{N-1} r_k^2 \leq 3 \sum_{k=0}^{N-1} |r_k| \leq 3 \sum_{k=0}^{N-1} (w_{\thet_n,k} + w_{\thet_n, N-1-k}) \leq 6 \sum_{k=0}^\infty w_{\thet_n, k} \leq \frac{CN}{n}.
    }
    The assertion of the lemma now follows. 
\end{proof}

Now we are ready to prove Theorem \ref{thm:main}.

\begin{proof}[Proof of Theorem \ref{thm:main}]
The first part of Theorem \ref{thm:main} follows from Proposition \ref{prop:solutions}. Below we provide the proof of the second part.

Recall that by $\lambda_1, \lambda_2, \ldots, \lambda_N$ we denoted the eigenvalues of $T_N$. By $\smash{u_j}$ we denote the corresponding eigenvectors of unit norm. By Theorem \ref{thm:main:2}, we know that $\lambda_j$ are all simple. Therefore, the eigenvectors $u_j$ form an orthonormal basis of $\R^N$. It follows that $\smash{\tilde v_n = \sum_{j=1}^N c_{n,j} u_j}$ for some $c_{n,j} \in \R$, where $\tilde v_n$ are the approximate eigenvectors defined in \eqref{eq:v:approx}. Moreover, $\smash{\lVert \tilde{v}_n \rVert^2 = \sum_{j=1}^N c_{n,j}^2}$ and $\smash{T_N \tilde v_n = \sum_{j=1}^N \lambda_j c_{n,j} u_j}$. Let $\lambda_{j(n)}$ be the eigenvalue nearest to $g(\thet_n)$. Then
\formula{
\lVert T_N \tilde v_n - g(\thet_n) \tilde v_n \rVert^2 &= \sum_{j=1}^N (\lambda_j - g(\thet_n))^2 c_{n,j}^2 \\ &\geq (\lambda_{j(n)} - g(\thet_n))^2 \sum_{j=1}^N c_{n,j}^2 = (\lambda_{j(n)} - g(\thet_n))^2 \lVert \tilde v_n \rVert^2.
}

Assume that $C_1 \leq  n \leq N$, where $C_1 \geq 2$ is the constant from Lemma \ref{lem:v:norm}. By Lemmas~\ref{lem:TN} and \ref{lem:v:norm}, there is a constant $C_2 >0$ such that
\formula[eq:l-G]{
|\lambda_{j(n)} - g(\thet_n)| \leq \frac{\lV T_N \tilde v_n - g(\thet_n) \tilde v_n \rV}{\lV \tilde v_n \rV} \leq \frac{C_2}{(2-\alpha) \alpha^{3/2}}\, \frac{1}{nN^{\alpha}} \biggl( 1 - \frac{n}{N+1} \biggr)
}
We claim that $\lambda_{j(n)}=\lambda_n$ for sufficiently large $n$. For $1 \leq n \leq N+1$, we define
\formula{
& L_n = \frac{(2n-1)\pi}{2(N+1)}, & L_n^* = \frac{(n-1)\pi}{N}, && U_n = \frac{n\pi}{N+1},& 
}
and for $1 \leq n \leq N$,
\formula{
M_n = \frac{U_n+L_{n+1}}{2} = \frac{(4n+1)\pi}{4(N+1)}.
}
Moreover, let
\formula{
K_\alpha = \biggl( \frac{C_2 C_3}{(2-\alpha)\alpha^{5/2}} \biggr)^{1/\alpha},
}
where $C_3>0$ is a sufficiently large constant such that $C_3 \geq C_1^{\alpha} C_2^{-1} (2-\alpha) \alpha^{5/2}$. Then, we have $K_\alpha \geq C_1$. We claim that $\lambda_{j(n)} \in (g(M_{n-1}), g(M_n))$ for every $n > K_\alpha$.

By the mean value theorem, there is $\xi_n \in (M_{n-1}, \thet_n)$ such that
\formula{
|g(\thet_n) - g(M_{n-1})| = |g'(\xi_n)| (\thet_n - M_{n-1}).
}
By \eqref{eq:thet:bounds}, we have $U_{n-1} \leq M_{n-1} \leq \xi_n \leq \thet_n \leq U_n \leq M_n$. In particular, $\xi_n \in (M_{n-1}, M_n)$. Note that for $2 \leq n \leq N$, we have  $\xi_n \approx n/N$ and $\pi - \xi_n \approx 1-n/(N+1)$. More precisely,
\formula{
& \frac{\pi}{4}\frac nN \leq \xi_n\leq \frac{9\pi}{8}\frac nN, & \frac{3\pi}{4}\biggl( 1 - \frac{n}{N+1}\biggr) \leq \pi-\xi_n\leq \frac{7\pi}{4}\biggl( 1 - \frac{n}{N+1}\biggr). &
}
Using Lemma \ref{lem:asymptotics}, the assumption that $n > K_\alpha$, and \eqref{eq:l-G}, we find that
\formula{
|g(\thet_n) - g(M_{n-1})| &\geq |g'(\xi_n)| (L_n - M_{n-1}) \\
&= \tfrac{1}{2} |g'(\xi_n)| (L_n - U_{n-1}) \\
&\geq  C \alpha  \biggl( 1-\frac{n}{N+1} \biggr) \biggl(\frac{n}{N}\biggr)^{\alpha-1} \frac{\pi}{2(N+1)} \\
&\geq \frac{\alpha n^\alpha}{C_3 n N^\alpha} \biggl( 1-\frac{n}{N+1} \biggr) \\
&\geq\frac{C_2}{(2-\alpha) \alpha^{3/2}}\, \frac{1}{nN^{\alpha}} \biggl( 1 - \frac{n}{N+1} \biggr) \geq |\lambda_{j(n)} - g(\thet_n)|,
}
whenever $C_3$ is large. An analogous argument shows that
\formula{
|g(\thet_n) - g(M_n)| \geq |\lambda_{j(n)} - g(\thet_n)|
}
for sufficiently large $C_3$. It follows that $\lambda_{j(n)} \in (g(M_{n-1}), g(M_{n}))$ for $n > K_\alpha$.

Clearly, the intervals $(g(M_{n-1}), g(M_n))$ are mutually disjoint. By Theorem \ref{thm:main:2}, we have
\formula{
\lambda_1 < \lambda_2 < \ldots < \lambda_{\lfloor K_\alpha \rfloor} \leq g(U_{\lfloor K_\alpha \rfloor}) < g(M_{\lfloor K_\alpha \rfloor}).
}
Therefore, there are exactly $\lfloor K_\alpha \rfloor$ eigenvalues not included in the above class. It follows that $j(n)=n$ for all $n > K_\alpha$. This completes the second part of Theorem \ref{thm:main} when $K_\alpha < n \leq N$. Clearly, if $N\leq \floor{K_\alpha}$, there are no indices $n>K_\alpha$, and we proceed directly to the case $n\leq K_\alpha$ below.

Now, let us assume that $2 \leq n \leq K_\alpha$. By Theorem \ref{thm:main:2} and \eqref{eq:thet:bounds}, both $\lambda_n$ and $g(\thet_n)$ belong to $[g(L_n^*), g(U_n)]$. Again by the mean value theorem, we deduce that there is a $\xi_n \in (L_n^*, U_n)$, such that
\formula[eq:smalln]{
|\lambda_n - g(\thet_n)| &\leq g(U_n)-g(L^*_n) = g'(\xi_n) (U_n - L_n^*).
}
Since
\formula[eq:lem:warunki:2]{
& \frac \pi2 \, \frac{n}{N} \leq \xi_n \leq \pi \, \frac{n}{N}, & \pi \biggl( 1- \frac{n}{N+1} \biggr) \leq \pi - \xi_n \leq 2\pi \biggl( 1- \frac{n}{N+1} \biggr),
}
we have $\xi_n \approx n/N$ and $\pi-\xi_n \approx 1-n/(N+1)$. Using Lemma \ref{lem:asymptotics} with \eqref{eq:smalln} yields
\formula{
|\lambda_n - g(\thet_n)| &\leq C \alpha  \biggl( 1-\frac{n}{N+1} \biggr) \biggl(\frac{n}{N}\biggr)^{\alpha-1} \frac{\pi}{N} \biggl(1 -\frac{n}{N+1}\biggr) \\
&\leq \frac{C_4 \alpha n^\alpha }{nN^{\alpha}} \biggl( 1 - \frac{n}{N+1} \biggr)^2 \\
&\leq \frac{C_2 C_3 C_4}{(2-\alpha)\alpha^{3/2}}\, \frac{1}{n N^\alpha} \biggl( 1-\frac{n}{N+1}\biggr)
}
for some $C_4>0$. In the last line we used our assumption that $n \leq K_\alpha$.

For completeness, we will prove the theorem for $n=1$. By Theorem \ref{thm:main:2} and \eqref{eq:thet:bounds}, both $\lambda_1$ and $g(\thet_1)$ belong to $[0, g(U_1)]$. It follows that
\formula{
|\lambda_1 - g(\thet_1)| \leq g(U_1) \leq \frac{C}{N^\alpha} \leq \frac{C}{(2-\alpha)\alpha^{3/2}} \, \frac{1}{N^\alpha} \biggl( 1-\frac{1}{N+1}\biggr),
}
since $1-1/(N+1)\geq1/2$ and $(2-\alpha)\alpha^{3/2}$ is bounded for $\alpha\in(0,2)$. This completes the proof of the second part of Theorem \ref{thm:main}.
\end{proof}

\subsection{Approximation formulas}

In this subsection, we prove Theorems \ref{thm:main:4} and \ref{thm:main:3}.

\begin{proof}[Proof of Theorem~\ref{thm:main:4}]
    Following the notation from the proof of Theorem \ref{thm:main}, let us denote $U_n = n\pi/(N+1)$. Since $\thet_n$ is the solution of the phase equation \eqref{eq:arg:eq}, we have $U_n - \thet_n = \eta(\thet_n)/(N+1)$. By \eqref{eq:thet:bounds} and Proposition~\ref{lem:theta:plus}, we have $\eta(\thet_n)>0$, and hence $\theta_n<U_n$ for every $1 \leq n \leq N$. By the mean value theorem, there is $\xi_n \in (\thet_n, U_n)$ satisfying
    \formula{
    g(\thet_n) &= g(U_n) - g'(\xi_n)(U_n - \thet_n) \\
    &= g(U_n) - g'(U_n)(U_n - \thet_n) + (g'(U_n) - g'(\xi_n))(U_n - \thet_n).
    }
    Applying the mean value theorem once again, this time to $\eta$, we find that
    \formula{
    \eta(\thet_n) = \eta(U_n) - \eta'(\zeta_n)(U_n - \thet_n)
    }
    for some $\zeta_n \in (\thet_n, U_n)$. Thus,
    \formula[eq:i0]{
    g(\thet_n) = g(U_n) - \frac{g'(U_n) \eta(U_n)}{N+1} + E_3(n,N),
    }
    where
    \formula{
    E_3(n,N) = \frac{g'(U_n)\eta'(\zeta_n)(U_n-\thet_n)}{N+1} + (g'(U_n) - g'(\xi_n))(U_n - \thet_n).
    }
    Recall that $g$ is increasing and $\eta$ is decreasing, hence we estimate
    \formula[eq:i1]{
    |E_3(n,N)| \leq \biggl( \frac{g'(U_n)|\eta'(\zeta_n)|}{N+1} + |g'(U_n) - g'(\xi_n)| \biggr) (U_n - \thet_n)
    }

    Since $U_n \approx n/N$ and $\pi - U_n \approx 1-n/(N+1)$, by Lemma \ref{lem:asymptotics}, we have $g'(U_n) \approx \alpha(n/N)^{\alpha-1}(1-n/(N+1))$. By Proposition \ref{lem:theta:plus}, we also have $|\eta'(\zeta_n)| \leq (2-\alpha)/2$. To estimate the second term in \eqref{eq:i1}, we use the mean value theorem yet again. There is $\sigma_n \in (\xi_n , U_n)$ such that
    \formula{
    |g'(U_n) - g'(\xi_n)| = |g''(\sigma_n)| (U_n - \xi_n).
    }
    Direct differentiation gives $g''(\sigma_n) = \alpha (2\sin(\sigma_n/2))^{\alpha-2}(\alpha \cos^2 (\sigma_n/2)-1)$. Hence, $|g''(\sigma_n)| \leq C \alpha  (\sin \frac {\sigma_n}{2})^{\alpha-2}$. Recall that $\thet_n \leq \xi_n \leq \sigma_n \leq U_n$, which implies that $\sigma_n \approx n/N$. Consequently, $|g''(\sigma_n)| \leq C \alpha  ( n/N )^{\alpha-2}$. Moreover, by \eqref{eq:thet:bounds},
    \formula{
    U_n - \xi_n \leq U_n - \thet_n \leq U_n - \frac{(2n-1)\pi}{2(N+1)} = \frac{\pi}{2(N+1)} \leq \frac{C}{N}.
    }
    Using the above estimates, we arrive at
    \formula[eq:i5]{
    |E_3(n,N)| \leq C \alpha  \biggl( \frac{n^{\alpha-1}}{N^{\alpha+1}}\biggl(1- \frac{n}{N+1}\biggr) + \frac{n^{\alpha-2}}{N^{\alpha}} \biggr) \leq C \alpha  \frac{n^{\alpha-2}}{N^\alpha}.
    }
    Combining \eqref{eq:i0} and \eqref{eq:i5} with Theorem \ref{thm:main} yields
    \formula{
    \lambda_n = g(U_n) - \frac{g'(U_n) \eta(U_n)}{N+1} + E_1(n,N) + E_3(n,N),
    }
    where $E_1(n,N)$ is defined in Theorem \ref{thm:main} and $E_3(n,N) = \mathscr{O}(\alpha n^{\alpha-2} N^{-\alpha})$.
\end{proof}

\begin{proof}[Proof of Theorem \ref{thm:main:3}]
    
We start with the first part of the theorem. Assume that $n/N \leq 1/2$. By \eqref{eq:thet:bounds}, we have $\thet_n \leq \pi/2$. Proposition \ref{prop:asymptotics:small:mu} allows us to rewrite the phase equation as follows:
\formula{
(N+1)\thet_n + \frac{(2-\alpha)\pi}{4} - \frac{2-\alpha}{2} \thet_n + E_0(\thet_n) = n\pi,
}
where $E_0$ is defined in \eqref{eq:delta}.
Solving this for $\thet_n$ yields
\formula{
\thet_n = \frac{(4n - 2 +\alpha)\pi - 4E_0(\thet_n)}{4N+2\alpha}.
}
By the mean value theorem, there is $\xi_n \in (\thet_n, (4n - 2 +\alpha)\pi/(4N+2\alpha))$ such that
\formula{
g(\thet_n) = g\biggl(\frac{(4n - 2 +\alpha)\pi}{4N+2\alpha}\biggr) - \frac{4 g'(\xi_n) E_0(\thet_n)}{4N+2\alpha}.
}
We denote the error component above as $E_4(n,N)$. As before, we have $\xi_n \approx n/N$ and $\pi - \xi_n \approx 1-n/(N+1)$. Using Lemma \ref{lem:asymptotics} together with the fact that $g'>0$ on $(0, \pi)$, we estimate
\formula[eq:Delta]{
 0< E_4(n,N)  &=  \frac{4g'(\xi_n) E_0(\thet_n)}{4N+2\alpha} \leq C \alpha  \biggl(1-\frac{n}{N+1}\biggr) \, \frac{n^{\alpha-1}}{N^{\alpha}} \, E_0\biggl(\thet_n\biggr).
}
The required bounds for $E_4(n,N)$ now follow from \eqref{eq:delta} and $\thet_n\approx n/N$. For $\alpha=1$, we also use $\thet_n^2\log(e/\thet_n)\leq C(n/N)^2\log(N/n)$, since $n/N\leq1/2$.
With \eqref{eq:Delta} it is easy to verify that $E_4(n,N)$ is of the desired form.

The proof of the second part of the theorem is analogous. For $n>N/2$, \eqref{eq:thet:bounds} gives $\thet_n\geq\pi/4$, so Proposition~\ref{prop:eta:big} applies. It follows that
\formula{
g(\thet_n) = g\biggl( \frac{(n-C_\alpha)\pi}{N+1-C_\alpha}\biggr) - \frac{g'(\xi_n)E_\pi(\thet_n)}{N+1-C_\alpha},
}
where $C_\alpha$ and $E_\pi(\thet)$ are defined in Proposition \ref{prop:eta:big}, and $\xi_n \in (\thet_n, \pi(n-C_\alpha)/(N+1-C_\alpha))$. In particular,
\formula{
C_\alpha = \frac{1}{\pi} \int_0^\infty \frac{t}{(1+t^2)^{3/2}} \, p_\alpha(t) \,dt \leq \int_0^1 \frac{t}{(1+t^2)^{3/2}} \, dt + \frac{1}{2} \int_1^\infty \frac{t}{(1+t^2)^{3/2}} \, dt = 1- \frac{1}{2\sqrt{2}},
}
since the argument of the $\arccot$ function in $p_\alpha(t)$ is positive for $t>1$. We denote the error term above as $E_5(n,N)$ and repeat the arguments used in the previous case to obtain
\formula{
0 < E_5(n,N) \leq C \alpha  \frac{n^{\alpha-1}}{N^\alpha} \biggl( 1-\frac{n}{N+1}\biggr)^4.
}
This completes the proof.
\end{proof}

Finally, we prove Corollary \ref{cor:equation}.

\begin{proof}[Proof of Corollary~\ref{cor:equation}]
    Assume that $2 \leq n \leq N$ and recall the phase equation \eqref{eq:arg:main}:
    \formula[eq:ae:1]{
    (N+1) \thet_n + \eta(\thet_n) = n\pi.
    }
    Let $s_n = g^{-1}(\lambda_n)$. We are looking for the error term $E_2(n,N)$ in the following equation:
    \formula[eq:ae:2]{
    (N+1) s_n + \eta(s_n) = n\pi + E_2(n,N).
    }
    Subtracting \eqref{eq:ae:1} from \eqref{eq:ae:2} yields
    \formula{
    E_2(n,N) = (N+1)(s_n - \thet_n) + (\eta(s_n)-\eta(\thet_n)).
    }
    By the mean value theorem applied to $\eta$, we find that there is $\zeta_n$ between $\thet_n$ and $s_n$ such that
    \formula{
    \eta(s_n) - \eta(\thet_n) = \eta'(\zeta_n) (s_n - \thet_n).
    }
    It follows that
    \formula[eq:E]{
    E_2(n,N) = (N+1+\eta'(\zeta_n))(s_n-\thet_n).
    }
    To estimate the error, we will use Theorem \ref{thm:main}. Note that \eqref{eq:main} is equivalent to
    \formula{
    s_n = g^{-1}(\lambda_n) = g^{-1}(g(\thet_n) + E_1(n,N)).
    }
    The mean value theorem applied to $g$ implies that there is $\xi_n$ between $\thet_n$ and $s_n$ such that
    \formula{
    s_n - \thet_n = \frac{ E_1(n,N)}{g'(\xi_n)}.
    }
    By Theorem \ref{thm:main:2} and \eqref{eq:thet:bounds}, both $s_n$ and $\thet_n$ belong to the interval $[L_n^*, U_n]$, as does $\xi_n$. By \eqref{eq:lem:warunki:2}, $\xi_n \approx n/N$ and $\pi - \xi_n \approx 1-n/(N+1)$. By Lemma \ref{lem:asymptotics}, we have $g'(\xi_n) \approx \alpha (n/N)^{\alpha-1} (1-n/(N+1))$. By this and the definition \eqref{eq:g} of $E_1(n,N)$, we deduce that
    \formula[eq:difference]{
    s_n - \thet_n = \mathscr{O}\biggl( \frac{1}{(2-\alpha)\alpha^{5/2}} \, \frac{n^{-1} N^{-\alpha} (1-n/(N+1))}{(n/N)^{\alpha-1} (1- n/(N+1))}\biggr) = \mathscr{O}\biggl( \frac{1}{(2-\alpha) \alpha^{5/2}} \, \frac{1}{n^\alpha N} \biggr).
    }
    By Proposition~\ref{lem:theta:plus}, $0<N+\alpha/2\leq N+1+\eta'(\zeta_n)\leq N+1$. Combining this with \eqref{eq:E} and \eqref{eq:difference}, we obtain
    \formula{
    |E_2(n,N)|\leq(N+1)|s_n-\thet_n|\leq\frac{C}{(2-\alpha)\alpha^{5/2}n^\alpha},
    }
    as desired.
    
    We treat $n=1$ separately, interpreting $\eta(0)$ by continuity. By Theorem~\ref{thm:main:2} and Proposition~\ref{lem:theta:plus}, $0\leq s_1\leq\pi/(N+1)$ and $0\leq\eta(s_1)\leq(2-\alpha)\pi/4$. Hence,
    \formula{
    -\pi\leq E_2(1,N)=(N+1)s_1+\eta(s_1)-\pi\leq\frac{(2-\alpha)\pi}{4}.
    }
    Thus, $|E_2(1,N)| \leq \pi\leq C/((2-\alpha)\alpha^{5/2})$, since $(2-\alpha)\alpha^{5/2}$ is bounded on $(0,2)$.

    We now verify that $\eta$ coincides with $\eta_{\mathrm{BG}}$ appearing in \cite{bg2022}. Fix $\alpha \in (1,2)$ and $\thet \in (0,\pi)$. Set $n_N=\lfloor(N+1)\thet/\pi\rfloor$. For all sufficiently large $N$, we have $1 \leq n_N\leq N$ and $n_N/N \to \thet/\pi$. By Theorem~\ref{thm:main:2},
    \formula{
    \frac{(n_N-1)\pi}{N}\leq s_{n_N}\leq\frac{n_N\pi}{N+1}.
    }
Both endpoints converge to $\thet$, and hence $s_{n_N}\to\thet$. Moreover, $n_N\geq\eps N$ for some $\eps>0$ and sufficiently large $N$. The first part of the corollary implies that
\formula{
(N+1)s_{n_N}+\eta(s_{n_N})=n_N\pi+E_2(n_N,N)
}
Theorem~1(ii) in \cite{bg2022}, applied to the same eigenvalues, gives
\formula{
(N+1)s_{n_N}+\eta_{\mathrm{BG}}(s_{n_N})=n_N\pi+E_6(N),
}
where $E_6(N) = o(N^{1-\gamma})$ for some $\gamma \in (1,\alpha)$. Subtracting these identities yields
\formula{
\eta(s_{n_N})-\eta_{\mathrm{BG}}(s_{n_N})=E_2(n_N, N)-E_6(N).
}
For $N \to \infty$, the right-hand side converges to zero. Both $\eta$ and $\eta_{\mathrm{BG}}$ are continuous on $[0,\pi]$ and $s_{n_N}\to\thet$. It follows that $\eta(\thet)=\eta_{\mathrm{BG}}(\thet)$. As $\thet\in(0,\pi)$ was arbitrary, the functions coincide on $[0,\pi]$.
\end{proof}

\subsection{Numerical examples}
\label{sec:simulations}

We illustrate the results by comparing four eigenvalue approximations for $N=1000$ and $\alpha\in\{1/2,1,3/2\}$. Numerical comparisons of relative errors for the two-term approximation from Theorem~\ref{thm:main:4} were also presented in \cite{bg2022}, for different matrix sizes and $\alpha \in (1,2)$. Here we additionally consider the approximation obtained from the phase equation in Theorem \ref{thm:main} and the two explicit approximations at the spectral edges, including the cases $\alpha=1/2$ and $\alpha=1$, based on Theorem \ref{thm:main:3}.

\begin{figure}
\centering
\includegraphics[width = \textwidth]{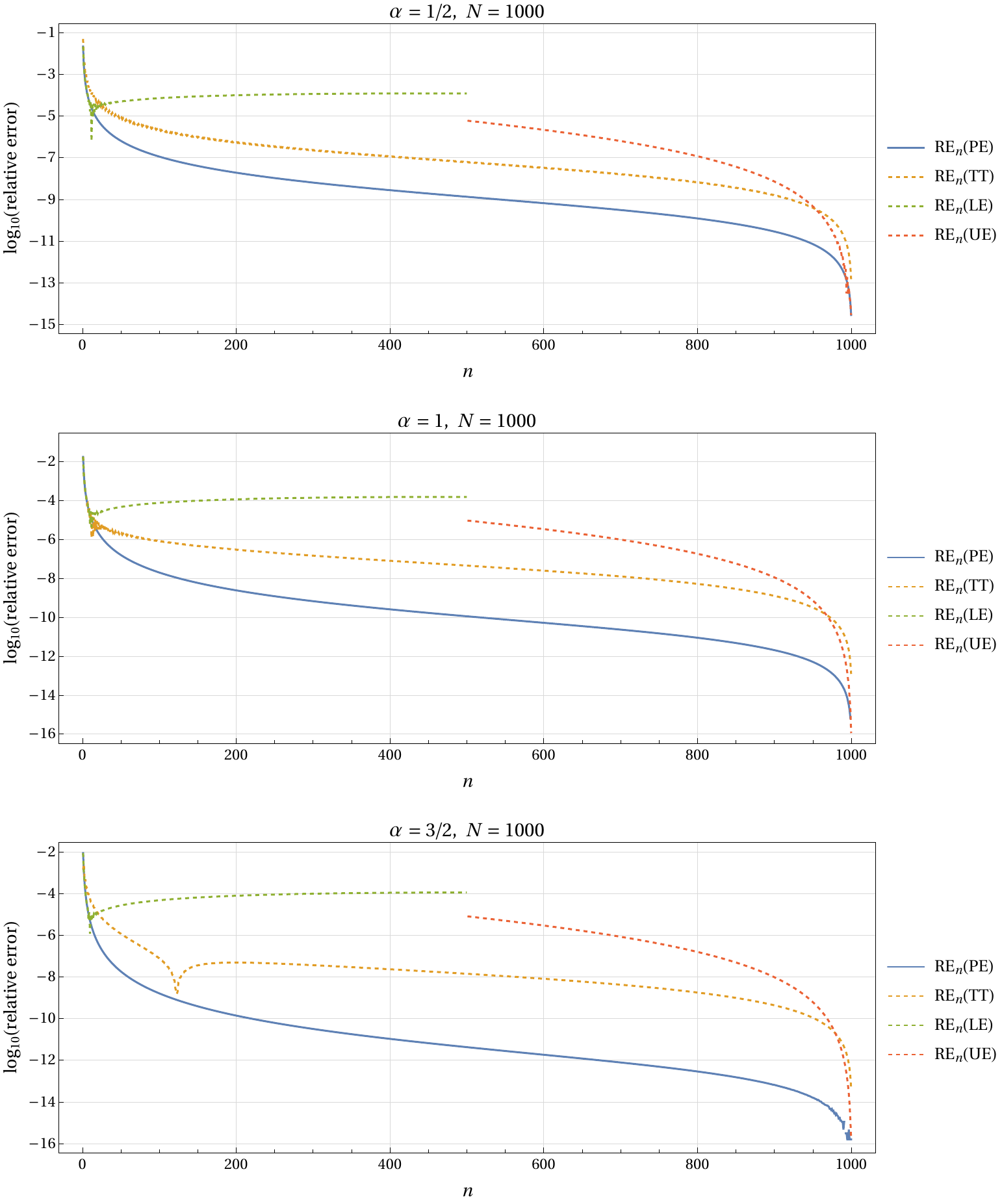}
\caption{Base-10 logarithms of the relative errors for $N=1000$ and $\alpha=1/2$, $1$, and $3/2$ (from top to bottom).}
\label{fig:relative-errors}
\end{figure}

With $U_{n,N}=n\pi/(N+1)$, we denote these approximations by
\formula{
\lambda_n^{\mathrm{PE}} &= g(\thet_n), \\ \lambda_n^{\mathrm{TT}} &= g(U_{n,N})-\frac{g'(U_{n,N})\eta(U_{n,N})}{N+1}, \\ \lambda_n^{\mathrm{LE}} &= g\biggl(\frac{(4n-2+\alpha)\pi}{4N+2\alpha}\biggr), \\ \lambda_n^{\mathrm{UE}} &= g\biggl(\frac{(n-C_\alpha)\pi}{N+1-C_\alpha}\biggr).
}
The approximations PE (phase equation) and TT (two-term) are considered for all $1\leq n\leq N$, while LE (lower edge) and UE (upper edge) are restricted to $n\leq N/2$ and $n>N/2$, respectively. For each approximation, we define the relative error by
\formula{
& \mathrm{RE}_n(A)=\frac{|\lambda_n-\lambda_n^A|}{\lambda_n}& \text{for } A\in\{\mathrm{PE},\mathrm{TT},\mathrm{LE},\mathrm{UE}\}.&
}
Figure~\ref{fig:relative-errors} displays $\log_{10}\mathrm{RE}_n(A)$ over the corresponding index ranges. The reference eigenvalues were computed numerically from $T_N$ in Wolfram Mathematica. The integrals defining $\eta$ and $C_\alpha$ were evaluated by Gaussian quadrature, and the phase equation was solved by fixed-point iteration. All computations used machine precision.

In the examples considered, PE is substantially more accurate than TT over most of the spectrum. For instance, at $n=500$, its relative errors are approximately $1.3\times10^{-9}$, $1.1\times10^{-10}$, and $4.2\times10^{-12}$ for $\alpha=1/2$, $1$, and $3/2$, respectively, whereas the corresponding TT errors are approximately $6.0\times10^{-8}$, $4.5\times10^{-8}$, and $1.4\times10^{-8}$. This comparison illustrates the gain in accuracy obtained by retaining the implicit phase equation. The ordering is not universal: TT can be more accurate for some of the smallest eigenvalues.

The explicit edge approximations exhibit different behaviour. LE closely follows PE for the first few eigenvalues, but becomes less accurate further into the spectrum. UE improves markedly towards the upper edge and eventually becomes more accurate than TT. Thus, the edge formulas provide useful explicit alternatives near the corresponding endpoints, while PE gives accurate approximations across a broader range of indices.

Errors close to machine precision should be interpreted with caution because of rounding errors in the reference eigenvalues and the approximations.

\section*{}

\subsection*{Acknowledgments} I would like to thank my supervisor, Mateusz Kwaśnicki, for his support throughout the preparation of this paper, helpful suggestions and discussions on the method for constructing approximate eigenvectors.

The large language model \textit{GPT-6 Astra} (OpenAI) was used to assist with reviewing the mathematical arguments, checking references, and editing earlier drafts of this article.

\bibliographystyle{abbrv}
\bibliography{bibi}

\end{document}